\documentclass[11pt,reqno]{amsart}
\usepackage{graphicx,color}
\usepackage{amssymb,amscd,latexsym, mathrsfs, epsfig}
\usepackage[all]{xy}

\newcommand\PP{\mathbb P}
\newcommand\C{\mathbb C}

\newcommand\R{\mathbb R}
\newcommand\Z{\mathbb Z}

\newcommand{\z}{\zeta}
\newcommand{\eps}{\varepsilon}

\newcommand\Aut{\operatorname{Aut}}
\renewcommand\Im{\operatorname{Im}}

\newcommand\ms{\operatorname{MS}}

\newcommand{\bra}{\langle}
\newcommand{\ket}{\rangle}
\newcommand{\wed}{\wedge}
\newcommand{\del}{\partial}

\newcommand{\A}{\mathcal{A}}
\newcommand{\B}{\mathcal{B}}
\newcommand{\G}{\mathcal{G}}
\newcommand{\M}{\mathcal{M}}
\newcommand{\T}{\mathcal{T}}
\newcommand{\W}{\mathcal{W}}

\newcommand{\X}{\mathcal{X}}
\newcommand{\Xsf}{\mathcal{X}^{\operatorname{sf}}}

\renewcommand{\S}{\text{\fontfamily{phv}\selectfont\small S}}
\newcommand{\U}{\text{\fontfamily{phv}\selectfont\small U}}
\newcommand{\dt}{\operatorname{DT}}
\renewcommand{\sf}{\operatorname{sf}}
\newcommand{\inst}{\operatorname{inst}}

\renewcommand{\d}{\delta}
\newcommand{\m}{\gamma_m}

\renewcommand{\th}{\theta}
\newcommand{\om}{\omega}
\newcommand{\Om}{\Omega}

\makeatletter \@addtoreset{equation}{section} \makeatother

\newtheorem{thm}{Theorem}
\newtheorem{lem}[thm]{Lemma}
\newtheorem{cor}[thm]{Corollary}
\newtheorem{conj}[thm]{Conjecture}
\newenvironment{rmk}{\noindent\textbf{Remark}.}{\\}
\newenvironment{exm}{\noindent\textbf{Example}.}{\\}

\title[]{Joyce-Song wall-crossing as an asymptotic expansion}
\author{Jacopo Stoppa}
\date{}
\address{}
\email{}
\begin{document}
\begin{abstract} We conjecture that the Joyce-Song wall-crossing formula for Don\-aldson-Thomas invariants arises naturally from an asymptotic expansion in the field theoretic work of Gaiotto, Moore and Neitzke. This would also give a new perspective on how the formulae of Joyce-Song and Kontsevich-Soibelman are related. We check the conjecture in many examples.
\end{abstract} 
\maketitle 
\section{Introduction} 
Donaldson-Thomas invariants \cite{rich} are the virtual counts of Gieseker or Mumford semistable coherent sheaves with fixed Chern character $\alpha$ on a Calabi-Yau threefold $X$ with $H^1(\mathcal{O}_X) = 0$. A complete theory in this generality has been developed in \cite{js}. More generally (often conjecturally) one can replace $\operatorname{Coh}(X)$ with a suitable 3-Calabi-Yau category endowed with a Bridgeland stability condition $\sigma$. The main work in this direction is \cite{ks}. 

The virtual count $\dt(\alpha, \sigma)$ is then a locally constant function of a stability condition $\sigma$ with values in $\mathbb{Q}$. However when $\sigma$ crosses certain real codimension 1 subvarieties of  the space of stability conditions (the walls) the invariants $\dt(\alpha,\sigma)$ jump in a complicated, universal way. One way to understand this wall-crossing behaviour is the Joyce-Song formula, equation (78) in \cite{js}. 

Our starting point is an observation of Joyce in \cite{joy} page 58: ``The transformation laws for Calabi-Yau 3-fold invariants [...] will also be written in terms of sums over graphs, and the author believes these may have something to do with Feynman diagrams in physics". 

Explicitly, in Joyce-Song theory the wall-crossing is given by:
\begin{align}\label{joyceFormula}
\nonumber\dt(\alpha, \sigma_+) = \sum_{n \geq 1}\sum_{\alpha_1 + \dots + \alpha_n = \alpha}&\frac{(-1)^{n-1}}{2^{n-1}}\U(\alpha_1, \dots, \alpha_n; \sigma_{\mp})\\
&\cdot\sum_{\T}\prod_{\{i \to j\}\subset \T}(-1)^{\bra\alpha_i, \alpha_j\ket}\bra\alpha_i, \alpha_j\ket\prod_k\dt(\alpha_k, \sigma_-)
\end{align}
summing over \emph{effective decompositions} $\sum_{i} \alpha_i$ of the $K$-theory class $\alpha$ (weighted by the combinatorial coefficients $\U$) and \emph{ordered trees} $\T$ (with vertices labelled by $\{1, \dots, n\}$). The brackets here denote the Euler form. The details are explained e.g. in \cite{js} Section 5. The coefficients $\U(\alpha_1, \dots, \alpha_n; \sigma_{\mp})$ are complicated functions of the cohomology classes $\alpha_i$ and of the slopes $\mu^{\pm}(\alpha)$ (or of some analogue notion, e.g. central charges $Z_{\pm}(\alpha_i)$ or reduced Hilbert polynomials $p(\alpha_i)$). Determining these coefficients is the main practical difficulty in applying the formula \eqref{joyceFormula}.

Naively the formula \eqref{joyceFormula} seems to be at odds with Joyce's remark: after all there is nothing here that could play the role of a coupling constant, and explicit examples show that the contributions of trees of different sizes may all have the same magnitude, with a lot of cancellation occurring, see e.g. \cite{me}, \cite{mps}.

The purpose of this paper is to point out that a possible solution to this puzzle, which is valid at least in the context of many examples originating from physical theories, follows naturally from the work of  Gaiotto, Moore an Neitzke \cite{gmn}. Before we explain roughly how this works in the rest of this introduction, we summarize the discussion below by the slogan that while \eqref{joyceFormula} is not itself an asymptotic expansion, it is the footprint of such an expansion, that is what remains of it when we approach a certain singular locus in the theory.

From a mathematical viewpoint the main object of study in \cite{gmn} is a set of (exponential, holomorphic) Darboux coordinates $\X_{\gamma}(\z)$ on a moduli space of singular Higgs bundles $\M$ (belonging to a certain class, which we will specify later in concrete examples). We always fix the gauge group $SU(2)$. The hyperk\"ahler metric constructed from $\X_{\gamma}(\z)$ is conjecturally the Hitchin metric $g_R$ (depending on a positive parameter $R$). As we will recall, while in general there is no closed formula for these coordinates, there exists however a natural asymptotic expansion for $\X_{\gamma}(\z)$ (equation \eqref{perturbative} below) around the so-called semiflat coordinates $\Xsf_{\gamma}(\z)$, the expansion parameter being the volume $R^{-1}$ of the fibres of the Hitchin fibration $\det\!: \M \to \B$, as $R \to \infty$ (where $\mathcal{B}$ is an affine space of meromorphic quadratic differentials). The terms in this asymptotic expansion are indexed by labelled trees $\T$, and the contribution of a tree $\T$ with $n$ vertices at generic points of $\M$ is of order less than $e^{-n R C}$, as $R \to \infty$ (for a certain constant $C > 0$). 

Donaldson-Thomas type invariants in this context arise from the physically defined BPS spectrum $\Omega(\gamma; u)$, a locally constant $\Z$-valued function on $\B$. One can then make a formal definition $\dt(\alpha; u) := \sum_{k \geq 1} \frac{\Omega(\alpha/k; u)}{k^2}$. A precise mathematical definition of the BPS spectrum $\Omega(\alpha; u)$, as well as the identification of the numbers $\dt(\alpha; u)$ with suitable Donaldson-Thomas invariants, is the object of much current investigation (in particular work in progress of Bridgeland and Smith \cite{smith}). As we will briefly mention, the heuristic geometric interpretation of the numbers $\Omega(\gamma; u)$ is that they enumerate special trajectories of the quadratic differential $\lambda^2(u)$  (or rather of any of its rotations $e^{i\theta}\lambda^2(u)$), representing the homology class $\gamma$. Moreover in all the examples we will consider the BPS spectrum $\Omega(\gamma; u)$ could be defined rigorously in terms of semistable representations of a suitable quiver associated with $\M$. But for most of the time in this paper we will leave aside these deeper aspects, and concentrate only on the wall-crossing behaviour of these invariants. Eventually we will arrive at a conjecture which is independent of the particular formulation of \cite{gmn}.

Now the leading corrections (of order $e^{-R C}$) to $\Xsf_{\gamma}(\z)$ are easily determined in terms of the BPS spectrum $\Omega(\gamma; u)$. However the estimate $e^{-n R C}$ for the contribution of a tree $\T$ with $n$ vertices is only valid away from a codimension $1$ subset $\ms \subset \B$, the so-called wall of marginal stability. Indeed in general the contribution of $\T$ has a jump across $\ms$ which is of leading order $e^{-R C}$. Continuity of the hyperk\"ahler metric requires cancellation, and so the existence of a corresponding leading order correction across the wall. We conjecture, and prove in a number of examples, that this procedure yields the Joyce-Song wall-crossing formula \eqref{joyceFormula}. Let $\mathcal{T}'$ be a $\Gamma$-labelled tree.
\begin{conj}\label{mainConj} The total contribution to wall crossing given by all the choices of a \emph{root} for $\mathcal{T}'$ in the Gaiotto-Moore-Neitzke asymptotic expansion\footnote{In general this contribution is only well defined up to certain singular integrals, and it is necessary to supplement the conjecture with one about their behaviour. We will give a precise statement at the end of Section \ref{Nf2Sect}.} \eqref{perturbative} matches the total contribution to the Joyce-Song formula \eqref{joyceFormula} given by all the possible \emph{orientations} of $\mathcal{T}'$.
\end{conj}
This approach explains why \eqref{joyceFormula} retains the structure of an asymptotic expansion (where the ``coupling constant" $R^{-1}$ has disappeared), and also offers an interpretation for the $\U$ functions in terms of certain integrals $\G_{\T}(\z)$. We mention some other points of interest of Conjecture \ref{mainConj}.
\begin{enumerate}
\item[$\bullet$] It seems striking that the Joyce-Song wall-crossing formula, which follows from the complicated theory of Ringel-Hall algebras, should emerge naturally from the GMN asymptotic expansion, which is obtained from a rather transparent superposition principle (the integral equation \eqref{tba} below) and a standard (at least for Physicists) asymptotic analysis.
\end{enumerate}
\begin{enumerate}
\item[$\bullet$] As a byproduct one would also obtain a new viewpoint on the equivalence of the wall-crossing formulae of Joyce-Song and Kontsevich-Soibelman \cite{ks}. Indeed the original motivation of GMN was to offer an interpretation for the latter formula. Quoting from \cite{pioline} ``[...] it should be noted that the Kontsevich-Soibelman wall-crossing
formula (and \emph{to a lesser extent}, the Joyce-Song formula) has already been derived or interpreted in various physical settings". Conjecture \ref{mainConj} would imply that in the GMN setting the interpretation of the two formulae is \emph{essentially the same}. 
\end{enumerate}
More precisely in \cite{gmn} the authors argue that the Kontsevich-Soibelman formula arises as a continuity condition for the holomorphic Darboux coordinates $\X_{\gamma}(\z)$, when one describes them as the solution of a suitable infinite-dimensional Riemann-Hilbert problem. The Riemann-Hilbert problem can be recast as an integral equation, \eqref{tba} below, which by standard arguments has the formal solution \eqref{perturbative}. So Conjecture \ref{mainConj} would lead to the following viewpoint: the Kontsevich-Soibelman formula follows simply from the existence of a continuous solution to the Riemann-Hilbert problem. If one actually tries to write down a solution using the asymptotic expansion \eqref{perturbative}, then the continuity condition becomes the Joyce-Song formula \eqref{joyceFormula}.  
\begin{enumerate}
\item[$\bullet$] Conjecture \ref{mainConj} could be a first step in addressing two additional important problems: comparing GMN theory with the works of Joyce \cite{joyHolo}, Bridgeland and Toledano-Laredo \cite{bt}; and using the recent motivic extension of GMN theory (see e.g. \cite{gmn3}) to describe a motivic extension of the Joyce-Song formula (or recover it, when available, as e.g. in the work of Chuang, Diaconescu and Pan \cite{cdp}).
\end{enumerate}
Notice that there are other conjectures in the literature which aim at comparing wall-crossing formulae obtained by physical arguments with those of Kontsevich-Soibelman and Joyce-Song, e.g. in the work of Manschot, Pioline and Sen \cite{mps}. Finally we should point out the papers of Chan \cite{chan} and Lu \cite{lu}, which also study \cite{gmn}, \cite{gmn2}, although with a completely different focus. While these works are concerned with the mirror-symmetric interpretation of GMN theory (in the local Ooguri-Vafa case for \cite{chan}, and much more ambitiously for moduli of singular $SU(2)$ Hitchin systems in \cite{lu}), we concentrate only on the asymptotic expansion \eqref{perturbative} and its connection with the formula \eqref{joyceFormula}.

The plan of the paper is the following: in Section \ref{general} we give a brief introduction to the basics of GMN theory (which we hope may be of independent interest), focusing on the class of examples which we will consider, namely the $SU(2)$ Seiberg-Witten gauge theories with $0 \leq N_f \leq 3$. Starting from Section \ref{instantonSection} we also present some computations involving the GMN connection and the GMN asymptotic expansion, which are implicit in \cite{gmn}, with the aim of explaining why Conjecture \ref{mainConj} could play an important role in comparing with \cite{joyHolo}, \cite{bt}. In Section \ref{basic} we explain Conjecture \ref{mainConj} in detail, checking it in many examples, and giving a purely combinatorial formulation at least for $N_f = 0$.

The reader who wants to get quickly to the computations with diagrams and integrals in Section 3 may want to look initially only at Sections \ref{bpsSect}, \ref{sfSect}, the first parts of \ref{instantonSection} and \ref{pertSection}, and \ref{wallcrossSect}.  

\noindent\textbf{Acknowledgements.} Many thanks to Tom Bridgeland, Tudor Dimofte, Heinrich Hartmann, Mart\'i Lahoz, Emanuele Macri, Sven Meinhardt, Ryo Ohkawa, Ivan Smith and Richard Thomas for useful discussions. The author is especially grateful to Daniel Huybrechts for his interest in this work, and to an anonymous Referee for pointing out a number of mistakes in an earlier version, as well as suggesting many improvements. This work was partially supported by the Hausdorff Center for Mathematics, Bonn and Trinity College, Cambridge. The research leading to these results has received funding from the European Research Council under the European Union's Seventh Framework Programme (FP7/2007-2013) / ERC Grant agreement no. 307119.
\section{Some general theory}\label{general}
\subsection{Connection to moduli of Higgs bundles} We concentrate for definiteness on the class of moduli spaces of singular $SU(2)$ Higgs bundles on $\PP^1$ considered e.g. in \cite{gmn2} Section 10 (see also \cite{tudor}). In the context of \cite{gmn} these correspond to the celebrated class of $SU(2)$ Seiberg-Witten gauge theories with $0 \leq N_f \leq 3$. Indeed while there is no doubt that the theory of \cite{gmn} applies much more generally, and that many interesting features and problems appear at higher genus, here we are only concerned in gathering enough motivation for our interpretation of \eqref{joyceFormula} as the footprint of an asymptotic expansion, and we believe that this is afforded already by this rather limited class of moduli spaces. In the rest of this paper $\M$ will always denote one of these $0 \leq N_f \leq 3$ moduli spaces (hopefully which one will be clear from the context).
\subsubsection{Standard Seiberg-Witten theory} The $N_f = 0$ case corresponds to moduli of pairs $(A, \varphi)$ of a $\mathfrak{su}(2)$ connection $A$ and Higgs fields $\varphi$ on $\PP^1$ which are singular at $z = 0, \infty$, with model singularity e.g. at $0$ given by (up to gauge transformations)
\begin{equation*}
\varphi \to -\frac{\Lambda}{|z|^{1/2}}
\left(\begin{matrix}
0            & 1 \\
e^{-i\theta} & 0
\end{matrix}\right)\frac{dz}{z},\,\,\, 
A \to \left(\begin{matrix}
-\frac{1}{8} & 0\\
0            & \frac{1}{8}
\end{matrix}\right)\left(\frac{dz}{z} - \frac{d\bar{z}}{\bar{z}}\right)
\end{equation*}
(here $\theta = \arg(z)$ and $\Lambda \in \C$ is a complex parameter, with $|\Lambda|$ corresponding to an energy scale). The Hitchin fibration $(A, \varphi) \mapsto \det\varphi$ maps $\mathcal{M}$ to an affine space of meromorphic quadratic differentials $\mathcal{B} \cong \C$, parametrized by 
\begin{equation*}
\lambda^2 = \left(\frac{\Lambda^2}{z^3} + \frac{2u}{z^2} + \frac{\Lambda^2}{z}\right) dz^2.
\end{equation*}
By the general theory the fibre of $\det$ over generic $u \in \B$ is a smooth elliptic curve, the Jacobian of the compactification $\overline{\Sigma}_u$ of 
\begin{equation*}
\Sigma_u := \{w^2 = \frac{\Lambda^2}{z^3} + \frac{2u}{z^2} + \frac{\Lambda^2}{z}\} \subset \C^2. 
\end{equation*}
So a smooth fibre $\M_u$ is (the Jacobian of) a double cover of $\PP^1$ ramified at $\{0, \infty, z^{\pm}_{tp}\}$, where
\begin{equation*}
z^{\pm}_{tp} = -\frac{u}{\Lambda} \pm \sqrt{\left(\frac{u}{\Lambda}\right)^2 - 1}.
\end{equation*}
(the ``turning points"). The turning points collide to $z^{\pm}_{tp} = \mp 1$ when $u = \pm \Lambda^2$. The corresponding fibres $\M_{\pm \Lambda^2}$ are nodal elliptic curves. There is a canonical meromorphic differential $\lambda = w dz$ on $\overline{\Sigma}_u$ for generic $u$, known as the Seiberg-Witten differential. The local system on $\mathcal{B}\setminus\{\pm \Lambda^2\}$ with stalks $H_1(\overline{\Sigma}_u, \Z)$ is denoted by $\Gamma$ and known as the charge lattice. $\Gamma$ is endowed with a nondegenerate skew-symmetric pairing $\bra \cdot, \cdot\ket$, coming from the intersection form on $H_1(\overline{\Sigma}_u, \Z)$. The crucial quantity for us is the central charge, an element of $\Gamma^* \otimes \C$ defined by integration of $\lambda$, 
\begin{equation*}
Z_{\gamma}(u) = \int_{\gamma}\lambda_u.
\end{equation*}
The wall of marginal stability $\ms \subset \B$ is the closure of the locus of $u \in \B$ for which $\{Z_{\gamma}(u) : \gamma \in \Gamma\} \subset \C$ is $1$-dimensional. It is a smooth real analytic curve inside $\B$, and in particular it contains the singular points $\pm \Lambda^2$ (because a generator of $H_1(\overline{\Sigma}_u, \Z)$ vanishes there). One can show that the curve $\ms \subset \B$ disconnects $\B$ into two components, a bounded one which we denote by $\B^s$, and an unbounded region denoted by $\B^w$ (physically, the bounded region corresponds to strong coupling, the unbounded one to weak coupling).
\subsubsection{$N_f = 1$} The $N_f = 1$ case corresponds to the meromorphic quadratic differentials
\begin{equation*}
\lambda^2 = \left(\frac{\Lambda^2}{z^3} + \frac{3u}{z^2} + \frac{2\Lambda m}{z} + \Lambda^2\right)dz^2
\end{equation*}
parametrized by $u \in \B \cong \C$ (for a fixed value of the complex parameter $m$). We will specialize to the most singular (and thus most interesting) situation when $m = 0$. Then the generic fibre $\M_u$ is a smooth elliptic curve, (the Jacobian of) a double cover of $\PP^1$ ramified at $z = 0$ and at the turning points, the three distinct roots of the cubic $\Lambda^2 z^3 + 3uz + \Lambda^2$ (notice that now $\infty$ is not a ramification point). There are three singular values of $u$, namely  $u^3 = \frac{1}{4}\Lambda^6$, for which two turning points collide; the singular fibres are again nodal elliptic curves. The definitions of the charge lattice and the central charge are unchanged. As above $\B$ splits into regions $\B^w, \B^s$ divided by the real analytic curve $\ms$.   
\subsubsection{Higher $N_f$} In the remaining cases the quadratic differentials are given by
\begin{equation*}
\lambda^2 = \left\{
\begin{matrix}
\left(\frac{\Lambda^2}{z^4} + \frac{2\Lambda m_1}{z^3} + \frac{4u}{z^2} + \frac{2\Lambda m_2}{z} + \Lambda^2\right)dz^2 \text{ for } N_f = 2,\\
\left(\frac{m^2_+}{z^2} + \frac{m^2_{-}}{(z-1)^2} + \frac{2\Lambda m + u}{2z} + \frac{2\Lambda m - u}{2(z-1)} + \Lambda^2\right)dz^2 \text{ for } N_f = 3.
\end{matrix}\right.
\end{equation*}
In the $N_f = 2$  case we will specialize $m$ to zero, giving the most singular case
\begin{equation*}
\lambda^2 = \left(\frac{\Lambda^2}{z^4} + \frac{4u}{z^2} + \Lambda^2\right)dz^2. 
\end{equation*}
The usual decomposition of $\mathcal{B}$ into $\mathcal{B}^s$ and $\mathcal{B}^w$ is still valid. This is also true for $N_f = 3$, for a suitable choice of the parameters $m, m_{\pm}$ (near $0$). 
\subsection{The Hitchin metric $g_R$} The smooth quasi-projective surface $\M$, as a moduli space of (singular) Higgs bundles, is endowed with a complete K\"ahler (in fact hyperk\"ahler) metric $g$ (a variation on the classical result of Hitchin \cite{hit}, see e.g. \cite{biqbo}). This is however not canonical, but comes naturally in a $1$-parameter family $g_R$ parametrized by $R > 0$. In other words we regard $\M$ as obtained by (infinite dimensional) hyperk\"ahler reduction of the hyperk\"ahler metric $\widetilde{g}_R$ on an affine space of pairs $(A, \varphi)$, namely 
\begin{equation*}
\widetilde{g}_R((\psi, \phi),(\psi, \phi)) = 2i \int\operatorname{Tr}(\psi^*\psi + R \phi\phi^*). 
\end{equation*}
where $\psi \in \Omega^{0, 1}(\mathfrak{sl}(2))$ is an infinitesimal gauge transformation, and $\phi \in \Omega^{1, 0}(\mathfrak{sl}(2))$ an infinitesimal Higgs field, so that the unitarity constraint in Hitchin's equations reads 
\begin{equation*}
F(A) + R[\varphi, \varphi^*] = 0.
\end{equation*} 
Notice that the above integral is well defined since the singularities of the connections and Higgs fields are fixed. The complex structure on $\M$ given by the moduli of Higgs bundles is independent of $R$, but the K\"ahler metric $g_R$ gives volume $R^{-1}$ to the smooth fibres of the Hitchin fibration. Also, while the $R$-dependence of $\widetilde{g}_R$ is straightforward, that of the hyperk\"ahler reduction $g_R$ is much more complicated and highly nonlinear. This $R$-dependence is the main object of study of \cite{gmn}, and $R^{-1}$ plays the role of the ``coupling constant" in the asymptotic expansion \eqref{perturbative} which we will use to interpret \eqref{joyceFormula}. One of the central conjectures of \cite{gmn} states that the $R$-dependence of $g_R$ is completely determined by a discrete invariant of $\mathcal{M}$, its BPS spectrum.
\subsection{BPS spectrum}\label{bpsSect} The BPS spectrum of $\M$ (counting BPS states) is a locally constant, $\Z$-valued function $\Omega(\gamma; u)$ on $\B \setminus\ms$ for $\gamma \in \Gamma$, with $\Omega(\gamma; u) = \Omega(-\gamma; u)$. The (countable) spectrum of BPS rays is defined by rays $\ell_{\gamma}(u) \subset \C$ spanned by the complex numbers $-Z_{\gamma}(u) \in \C$ where $\Omega(\gamma; u) \neq 0$. As we already mentioned in the introduction, a rigorous, a priori definition of BPS states and their counts starting from $\M$ is still lacking in general. In the special case of the $SU(2)$ Seiberg-Witten theories, the numbers $\Omega(\gamma; u)$ could be defined in terms of semistable representations of certain quivers naturally associated with $\M$. In all the examples we shall consider we will simply give a formula for $\Omega(\gamma; u)$ (indeed, for our purposes, we could give a working definition saying that a BPS spectrum for $\M$ is just a function $\Omega(\gamma; u)$ for which Conjecture \ref{gmnConj} below holds). But we should at least briefly mention the heuristic geometric interpretation of the numbers $\Omega(\gamma; u)$ emerging from \cite{gmn2}. Let us denote by $\lambda(u)$ the canonical (Seiberg-Witten) meromorphic 1-form on $\overline{\Sigma}_u$. We can also think of $\lambda$ as a 2-valued meromorphic differential on $\PP^1$. The rough idea is that $\Omega(\gamma; u)$ enumerates paths $\alpha\!: [0, 1] \to \PP^1$ which are solutions to 
\begin{equation*} 
\bra \lambda, \dot\alpha \ket \in e^{i\theta}\R^*
\end{equation*}
(for some angle $\theta \in S^1$), representing the homology class $\gamma$, and which are either closed or stretch between ramification points (the finite WKB curves of \cite{gmn2}).  
\subsubsection{Standard Seiberg-Witten} For $N_f = 0$, at strong coupling, it is possible to interpret all the BPS states in terms of two suitable paths (WKB curves) $\delta, \gamma_m$ joining the two turning points $z^{\pm}_{tp}$, such that $\delta - \gamma_m$ is an oriented $S^1$ around $z = 0$. Since they stretch between ramification points, one can regard these paths as closed paths in $\Sigma_u$, producing homology classes in $H_1(\overline{\Sigma}_u, \Z)$, still denoted by $\delta, \gamma_m$, with $\bra \delta, \gamma_m \ket = 2$ (reflecting that the paths $\delta, \gamma_m$ share both endpoints). In fact $\delta, \gamma_m$ are the vanishing cycles for the fibration $\{\overline{\Sigma}_u, u \in \C\}$. One can show that the only suitable WKB curves for $u \in \B^s$ are $\delta, \gamma_m$ (with a choice of orientation), so the BPS spectrum consists of just $\pm \delta, \pm \gamma_m$. The full BPS spectrum of $\M$ is given by
\begin{equation*}
\begin{matrix}
\Omega(\d; u) = \Omega(\m; u) = 1 \text{ for } u \in \B^s,\\
\Omega(k \d + (k + 1)\m) = \Omega((k+1)\d + k\m; u) = 1, \Omega(\d + \m; u) = -2 \text{ for } u \in \B^w,
\end{matrix}
\end{equation*}
for $k \geq 0$, plus the same indices for the negative of these charges. All the other indices vanish. The result at weak coupling can also be understood in terms of finite WKB curves, but is more complicated. In particular an infinite family of closed WKB curves appears, and one should make sense of counting these curves in a suitable way. (The approach taken in \cite{gmn2} is to enumerate them indirectly through their action of the Fock-Goncharov coordinates on $\mathcal{M}$, but we will not explain this further here). Alternatively one can just compute with one of the available wall-crossing formulae.
\subsubsection{$N_f = 1$} When $N_f > 0$ we encounter a new feature which we had kept silent up to now. Namely the spectrum $\Omega(\gamma; u)$ is not really a function on the homology local system $\Gamma$, but rather on an extension $\hat{\Gamma}$, with
\begin{equation*}
0 \to \Gamma_f \to \hat{\Gamma} \to \Gamma \to 0 
\end{equation*}
where $\Gamma_f$ is a rank $N_f$ local system. The fibre of the local system $\hat{\Gamma}$ is the sublattice of $H_1(\Sigma_u, \Z)$ (the open curve) spanned by vanishing cycles of the fibration $\{\Sigma_u, u \in \C\}$. The standard terminology is that $\hat{\Gamma}$ is really the charge lattice, while $\Gamma$ and $\Gamma_f$ are called respectively the gauge and flavour charge lattices. With this terminology in place we can write down the BPS spectrum. Let us denote by $\gamma_{1, 2, 3}$ the vanishing cycles in $H_1(\Sigma_u, \Z)$ (they can be realized on $\PP^1$ as the class of oriented segments joining two consecutive turning points around $z = 0$). Then we have 
\begin{equation*}
\begin{matrix}
\Omega(\gamma_1; u) = \Omega(\gamma_2; u) = \Omega(-\gamma_3; u) = 1 \text{ for } u \in \B^s,\\
\Omega((k+1)\gamma_2 + k \gamma_1 - k\gamma_3; u) = \Omega((k+1)\gamma_2 + k \gamma_1 - (k+1)\gamma_3; u) = 1\\
\Omega((k+1)\gamma_1 + k \gamma_2 - k\gamma_3; u) = \Omega((k+1)\gamma_1 + k \gamma_2 - (k+1)\gamma_3; u) = 1\\
\Omega(-\gamma_3; u) = \Omega(\gamma_1 + \gamma_2; u) = 1, \Omega(\gamma_1 - \gamma_3 + \gamma_2; u) = -2 \text{ for } u \in \B^w,
\end{matrix}
\end{equation*} 
for $k \geq 0$, plus the same indices for the negative of these charges. All the other indices vanish. Notice that in this case still writing $\gamma_{i},\, i = 1, 2, 3$ for the images in $\Gamma$ we have the single relation
\begin{equation*}
\gamma_1 + \gamma_2 + \gamma_3 = 0,
\end{equation*} 
and the intersection products are given by 
\begin{equation*}
\bra \gamma_1, \gamma_2\ket = \bra \gamma_2, \gamma_3\ket = \bra \gamma_3, \gamma_1\ket = 1. 
\end{equation*}
\subsubsection{$N_f = 2$} In this case $0, \infty$ are not ramification points, there are four turning points, and the charge lattice $\hat{\Gamma}$ is spanned by four vanishing cycles $\gamma^1_{1}, \gamma^1_{2}, \gamma^2_{1}, \gamma^2_2$ (which can be realized geometrically on $\PP^1$ as suitable paths joining two turning points). For the images in $\Gamma$ one has $\gamma^1_1 = \gamma^2_1, \gamma^1_2 = \gamma^2_2$, and $\bra \gamma^i_1, \gamma^j_2\ket = 1, \bra \gamma^i_1, \gamma^j_1\ket = \bra \gamma^i_2, \gamma^j_2\ket = 0$. The BPS spectrum is given by
\begin{equation*}
\begin{matrix}
\Omega(\gamma^1_1; u) = \Omega(\gamma^2_1; u) = \Omega(\gamma^1_2; u) = \Omega(\gamma^2_2; u) = 1 \text{ for } u \in \B^s,\\
\Omega(a^1_1\gamma^1_1 + a^2_1\gamma^2_1 + a^1_2 \gamma^1_2 + a^2_2 \gamma^2_2; u) = 1 (\text{for }a^1_1 + a^2_1 = k, a^1_2 + a^2_2 = k + 1, |a^1_i - a^2_i| \leq 1)\\
\Omega(a^1_1\gamma^1_1 + a^2_1\gamma^2_1 + a^1_2 \gamma^1_2 + a^2_2 \gamma^2_2; u) = 1 (\text{for }a^1_1 + a^2_1 = k+1, a^1_2 + a^2_2 = k, |a^1_i - a^2_i| \leq 1)\\
\Omega(\gamma^i_1 + \gamma^j_2; u) = 1, \Omega(\gamma^1_1 + \gamma^1_2 + \gamma^2_1 + \gamma^2_2; u) = -2 \text{ for } u \in \B^w,
\end{matrix}
\end{equation*} 
where $k \geq 0$, $a^i_j \geq 0$, plus the same indices for the negative of these charges. All the other indices vanish.
\subsubsection{$N_f = 3$} Again choosing $m, m_{\pm}$ suitably, $0, \infty$ are not ramification points, there are four turning points on $\PP^1\setminus\{0, \infty\}$, and five singular fibres for $\det$. There are four vanishing cycles $\gamma^{1,2,3,4}_1 \in \hat{\Gamma}$ with the same image in $\Gamma$, plus a vanishing cycle $\gamma_2$ with $\bra \gamma^i_1, \gamma_2\ket = 1$. As usual we have
\begin{equation*}
\Omega(\gamma^i_1; u) = \Omega(\gamma_2; u) = 1 \text{ for } u \in \B^s.
\end{equation*}
We do not write down the full BPS spectrum as strong coupling, but just notice that for $i \neq j$,
\begin{equation*}
\Omega(\gamma_2 + \gamma^i_1 + \gamma^j_1; u) = 1, \Omega(\gamma_2 + \sum_i \gamma^i_1; u) = -2 \text{ for } u \in \B^w.
\end{equation*}
\subsection{Hyperk\"ahler structure on $\M$} We now recall the conjectural description of the metric $g_R$ in terms on the BPS spectrum. For this we need to know a bit more about the hyperk\"ahler structure on $\M$. We will denote by $J_3$ the complex structure on $\mathcal{M}$ as a moduli space of Higgs bundles. There are two other (equivalent) complex structures $J_1, J_2$ that we can put on $\mathcal{M}$, induced by the actions on infinitesimal gauge transformations and Higgs fields given by 
\begin{equation*} 
\widetilde{J}_1(\psi, \phi) = (i\phi^*, -i\psi^*),\,\,\,\widetilde{J}_2(\psi, \phi) = (-\phi^*, \psi^*).
\end{equation*}
These satisfy the hyperk\"ahler condition
\begin{equation*}
J_1 J_2 = J_3,\,\,\,J_2 J_3 = J_1,\,\,\,J_3 J_1 = J_2.
\end{equation*}
We can form a whole $\PP^1$ of complex structures on $\mathcal{M}$ (known as the twistor sphere) parametrized by a coordinate $\z$,
\begin{equation*}
J(\z) = \frac{i(-\z + \bar{\z})J_1 - (\z + \bar{\z})J_2 + (1-|\z|^2)J_3}{1 + |\z|^2}.
\end{equation*}
For $\z \neq 0, \infty$, the $J(\z)$ are all equivalent to $J_1$, while for $\z = 0, \infty$ we recover $J_3$ (of course we need to rescale by $\z$ or $\z^{-1}$ to make sense of this). In fact for $\z \neq 0$ the map 
\begin{equation*}
(A, \varphi) \mapsto \mathscr{A}:=\frac{R}{\z}\varphi + A + R\z\varphi^* 
\end{equation*}
induces a biholomorphisms of $(\M, J(\z))$ with the moduli space of irreducible, meromorphic flat $\PP SL(2,\C)$ connections with prescribed singularities (this is a variation on the classical result of Donaldson \cite{don}, see e.g. \cite{biqbo}). Let us denote by $\om_i$ the symplectic forms obtained from $g_R$ and $J_i$ (more generally, we will write $\om(\z)$ for the symplectic form obtained combining $g_R$ and $J(\z)$). We will also write $\om_{\pm} = \om_1 \pm i \om_2$. Then one can show that 
\begin{equation*}
\varpi(\z) = -\frac{i}{2\z}\om_{+} + \om_3 -\frac{i}{2}\z \om_{-} 
\end{equation*} 
is a holomorphic symplectic form in complex structure $J(\z)$. For $\z = 0, \infty$ this is induced by the form
\begin{equation*}
\int\operatorname{Tr}(\phi_2\psi_1 - \phi_1\psi_2),
\end{equation*}
while for all other $\z$ it is induced by
\begin{equation*}
\int\operatorname{Tr}(\delta\mathscr{A}\wedge\delta\mathscr{A}).
\end{equation*}
From the form $\varpi(\z)$ we can reconstruct the metric $g_R$ uniquely. One of the key results of \cite{gmn} is a conjectural constuction of $\varpi(\z)$ in terms of the BPS spectrum, as an asymptotic expansion starting from a specific semiflat metric $g^{\sf}_R$ (flat on the fibres), with correction terms of order less than $e^{-R}$ as $R \to +\infty$. 

To keep the exposition and notation light, in the rest of this section we discuss this construction in the special case when $\M$ is the $N_f = 0$ moduli space. Two simplifications occur in this case:
\begin{enumerate}
\item[$\bullet$] the local systems $\hat{\Gamma}$ and $\Gamma$ coincide (i.e. $\Gamma_f$ is trivial);
\item[$\bullet$] the symplectic form $\bra-,-\ket$, restricted to the lattice spanned by $\delta$ and $\gamma_m$, is even.
\end{enumerate}
The first property leads to a mostly notational simplification; for the details of how to keep track of $\hat{\Gamma}$ see the Introduction to \cite{gmn2}. The second property allows one to get rid of all the sign issues in the definition of the holomorphic Darboux coordinates (related to the ``quadratic refinements" of \cite{gmn}). In section \ref{basic} these sign issues will become relevant, and we will show how to mend the definitions in this section to fit the $N_f > 0$ cases.
\subsection{The semiflat metric}\label{sfSect} In the following we will often assume we have fixed a local splitting of $\Gamma$ as $\Gamma^m\oplus\Gamma^e$, corresponding to a choice of symplectic basis for $H_1(\overline{\Sigma}_u, \Z)$, so that $\Gamma^e$, $\Gamma^m$ are spanned locally by $\gamma_e, \m$ with $\bra \gamma_e, \m \ket = 1$. In particular on a fixed fibre $\M_u = J(\overline{\Sigma}_u)$ we have dual angular coordinates $\th = (\th_e, \th_m)$, and we will write a point $m \in \M$ as $(u, \th)$, where $u \in \B$. More generally, we will write $\th_{\gamma}$ for the angular coordinate dual to $\gamma \in \Gamma$. 

A set of (exponential) local \emph{holomorphic Darboux coordinates} for the hyperk\"ahler metric on $\M$ is given by locally defined functions $\X_{\gamma}(m; \z)$ for $\gamma \in \Gamma$, $m \in \M$ and $\z \in \C^*$ such that
\begin{enumerate}
\item[$\bullet$] the function $\X_{\gamma}(m; \z)$ is holomorphic in the variable $\z$, at least in a nonempty dense open subset,\\
\item[$\bullet$] for $\gamma_1, \gamma_2 \in \Gamma$ we have $\X_{\gamma_1}\X_{\gamma_2} = \X_{\gamma_1 + \gamma_2}$,\\
\item[$\bullet$] $\X_{\gamma}(\z) = \overline{\X_{-\gamma}(-\bar{\z}^{-1})}$,
\end{enumerate}
and (what looks more like the Darboux property)
\begin{equation*}
\varpi(\z) = -\frac{1}{8\pi^2 R}\frac{d\X_e}{\X_e}\wedge\frac{d\X_m}{\X_m},
\end{equation*}
where $d$ denotes the differential on $\M$ (freezing the variable $\z$). The construction of \cite{gmn} produces (conjecturally) a set of distinguished exponential holomorphic Darboux coordinates for $g$, defined in terms of $\{\Om(\gamma; u)\}$. The starting point is a hyperk\"ahler metric $g^{\sf}$ on $\M\setminus\{\M_{\pm\Lambda^2}\}$, K\"ahler with respect to all the complex structures $J(\z)$, and semiflat, i.e. flat on the fibres of the Seiberg-Witten fibration (which is holomorphic in complex structure $J_3$). The metric $g^{\sf}$ is defined \emph{a priori} in terms of the putative holomorphic Darboux coordinates
\begin{equation}\label{sfCoords}
\Xsf_{\gamma}(u, \th; \z) := \exp\left(\pi R \z^{-1} Z_{\gamma}(u) + i\th_{\gamma} + \pi R\z \bar{Z}_{\gamma}(u)\right).
\end{equation}
The putative holomorphic symplectic form is of course
\begin{equation*}
\om^{\sf}(\z) = -\frac{1}{8\pi^2 R}\frac{d\Xsf_e}{\Xsf_e}\wedge\frac{d\Xsf_m}{\Xsf_m}.
\end{equation*}
One can check that this effectively defines a hyperk\"ahler structure on $\M\setminus\{\M_{\pm\Lambda^2}\}$ (with respect to the twistor sphere $J(\z)$). The holomorphic symplectic form is given by
\begin{equation*} 
\varpi^{\sf}(\z) = \frac{1}{4\pi}\left[\frac{i}{\z}\bra dZ, d\th\ket + \left(\pi R\bra dZ, d\bar{Z}\ket - \frac{1}{2\pi R}\bra d\th, d\th \ket\right) + i\z\bra d\bar{Z}, d\th\ket\right], 
\end{equation*}
where $\bra-,-\ket$ denotes the combination of the wedge product on forms with the symplectic form on $\Gamma^*\otimes\C$. The prospective K\"ahler form $\om_3$ is given by the $\z$-invariant part. In the standard notation in special geometry, one writes 
\begin{equation*}
a := Z_e(u),\,\,\,a_D := Z_m(u), \text{ and } \tau := \frac{\del a_D}{\del a},
\end{equation*}
from which
\begin{align*}
\om_3 &= \frac{i R}{2}\Im(\tau)\,da\wed d\bar{a} - \frac{1}{8\pi^2 R}d\th_e\wed d\th_m\\
&= \frac{i}{2}\left(R\Im(\tau)\,da\wed d\bar{a} + \frac{1}{4\pi^2 R}\Im(\tau)^{-1}dz\wed d\bar{z}\right), 
\end{align*}
where $dz = d\th_m - \tau d\th_e$ is only closed on the fibres. This is the customary expression for a semiflat metric in special geometry.\\  
\textbf{Example.} There is a local counterpart to this global semiflat metric, in the neighborhood of a singular fibre, by setting
\begin{align}
a &= Z_e(u) = u, \label{OVcentral1}\\
a_D &= Z_m(a) = \frac{1}{2\pi i}\left(a \log\frac{a}{\Lambda} - a\right)\label{OVcentral2},
\end{align}
and so 
\begin{equation*}
\tau = \frac{1}{2\pi i}\log\frac{a}{\Lambda}.
\end{equation*} 
\subsection{Instanton corrections}\label{instantonSection} We describe a model case of the main result and conjecture of \cite{gmn}, for the moduli space $\M$ we are considering. The authors propose a physical argument to the effect that there exist holomorphic Darboux coordinates $\X_{\gamma}(\z)$ for $g_R$, obtained as the unique solution to the integral equation 
\begin{equation}\label{tba}
\X_{\gamma}(\z) = \Xsf_{\gamma}(\z)\exp\left[-\frac{1}{4\pi i}\sum_{\gamma' \in \Gamma} \Omega(\gamma'; u)\bra\gamma, \gamma'\ket\int_{\ell_{\gamma'}}\frac{d\z'}{\z'}\frac{\z'+\z}{\z'-\z}\log(1 - \X_{\gamma'}(\z'))\right].
\end{equation}
More precisely, they prove the following result:
\begin{thm}[GMN] For $R$ large enough, iteration starting from the semiflat coordinates $\Xsf_{\gamma}(\z)$ converges to a solution $\X_{\gamma}(\z)$ of \eqref{tba}. The functions $\X_{\gamma}(\z)$ obtained is this way form a set of holomorphic Darboux coordinates for a hyperk\"ahler metric on $\M$. 
\end{thm}
The main conjecture is then:
\begin{conj}[GMN]\label{gmnConj} The equation \eqref{tba} admits a unique solution $\X_{\gamma}(\z, R)$, defined for all $R > 0$, such that the functions $\X_{\gamma}(\z, R)$ form a set of holomorphic Darboux coordinates for the Hitchin metric $g_R$.  
\end{conj}
\begin{rmk} The principal object of study in Joyce-Song theory are the Donaldson-Thomas invariants $\dt(\gamma; \sigma)$, while the BPS state counts $\Omega(\gamma'; \sigma)$ are only defined indirectly through the multi-cover formula
\begin{equation}\label{multiCover}
\dt(\gamma; \sigma) = \sum_{n >0, n | \gamma}\frac{\Omega(\gamma/n; \sigma)}{n^2}.
\end{equation}
In \cite{gmn} one has precisely the opposite situation: the basic quantity is the BPS spectrum $\{\Omega(\gamma; u)\}$, and the formal analogues of the Donaldson-Thomas invariants arise very naturally in the analysis of the integral equation \eqref{tba}, namely by considering a power series expansion
\begin{equation*}
-\sum_{\gamma' \in \Gamma} (\Omega(\gamma'; u)\log(1 - \X_{\gamma'}(\z'))\gamma' = \sum_{\gamma' \in \Gamma} f^{\gamma'}\X_{\gamma'}
\end{equation*} 
for certain coefficients $f^{\gamma'} \in \Gamma$. The unique solution is given by
\begin{equation}\label{fCoeff}
f^{\gamma} = \sum_{n > 0, \gamma = n\gamma'} \frac{\Omega(\gamma'; u)}{n}\gamma',
\end{equation}
and then the integral equation \eqref{tba} takes the more amenable form 
\begin{equation*}
\X_{\gamma}(\z) = \Xsf_{\gamma}(\z)\exp\bra \gamma, \frac{1}{4\pi i}\sum_{\gamma' \in \Gamma}  f^{\gamma'}\int_{\ell_{\gamma'}}\frac{d\z'}{\z'}\frac{\z'+\z}{\z'-\z} \X_{\gamma'}(\z')\ket.
\end{equation*}
Notice that if we define formally a set of numbers $\dt(\gamma; u)$ via \eqref{multiCover} we have in fact
\begin{equation*}
f^{\gamma} = \dt(\gamma)\gamma.
\end{equation*}
\end{rmk}
\begin{exm} In the $N_f = 0$ strong coupling region we have $\d = 2 \gamma_e - \m$, and the integral equations for $\X_{\gamma_e}, \X_{\m}$ are 
\begin{align*}
\X_{m}(\z) &= \Xsf_{m}(\z)\exp\left[-\frac{1}{4\pi i}(\mp 2)\int_{\ell_{\pm(2\gamma_e - \m)}}\frac{d\z'}{\z'}\frac{\z'+\z}{\z'-\z}\log(1 - \X^{\pm 2}_{e}(\z')\X^{\mp 1}_{m}(\z'))\right],\\
\X_{e}(\z) &= \Xsf_{e}(z)\exp\left[-\frac{1}{4\pi i}(\mp 1)\int_{\ell_{\pm(2\gamma_e - \m)}}\frac{d\z'}{\z'}\frac{\z'+\z}{\z'-\z}\log(1 - \X^{\pm 2}_{e}(\z')\X^{\mp 1}_{m}(\z'))\right.\\
&\left.-\frac{1}{4\pi i}(\pm 1)\int_{\ell_{\pm(\m)}}\frac{d\z'}{\z'}\frac{\z'+\z}{\z'-\z}\log(1 - \X^{\pm 1}_{m}(\z'))\right]\\
&= \Xsf_{e}(z)\frac{\X^{1/2}_{m}(\z)}{\Xsf_{m}(\z)}\exp\left[-\frac{1}{4\pi i}(\pm 1)\int_{\ell_{\pm(\m)}}\frac{d\z'}{\z'}\frac{\z'+\z}{\z'-\z}\log(1 - \X^{\pm 1}_{m}(\z'))\right],
\end{align*}
and so equivalent to a \emph{single} integral equation for $\X_{m}(\z)$.  
\end{exm}
\textbf{Example.} Again, the above global statements have a local counterpart around a singular fibre, which indeed provides important motivation for the ansatz \eqref{tba}. We consider the theory over a disc $\Delta \subset \C$ with radius $|\Lambda|$. The charge lattice is $\Gamma \cong \Z^2$ spanned by $\gamma_e, \gamma_m$, with $\bra \gamma_e, \gamma_m\ket = 1$. We pick the central charge given by the expressions \eqref{OVcentral1}, \eqref{OVcentral2}, \emph{but with \eqref{OVcentral2} rescaled by $q^2$ for some $q \geq 1$}, and declare a single BPS state with electric charge $q \in \mathbb{N}_{> 0}$, that is $\Omega(q\gamma_{e}; u) = \Omega(-q\gamma_{e}; u) = 1$ for all $u \in \Delta$. We set all the other BPS invariants to zero. So we are led to the equations
\begin{align*}
\X_{e} &= \Xsf_e,\\
\nonumber \X_{m} = \Xsf_m \exp &\left[\frac{i q}{4\pi}\int_{\ell_{\gamma_e}}\frac{d\z'}{\z'}\frac{\z'+\z}{\z'-\z}\log(1 - \X_e (\z')^q)\right.\\
&\left. -\frac{i q}{4\pi}\int_{\ell_{-\gamma_e}}\frac{d\z'}{\z'}\frac{\z'+\z}{\z'-\z}\log(1 - \X_e (\z')^{-q})\right].
\end{align*}
In a key computation in \cite{gmn} Section 4.3, the authors prove that these are holomorphic Darboux coordinates for the hyperk\"ahler metric first described by Ooguri and Vafa in \cite{ogvafa} (see also \cite{pelham} Section 3 for a detailed mathematical exposition), by comparing with the explicit form of this metric in Gibbons-Hawking ansatz. In the approach of \cite{gmn} these equations provide the basic clue for the integral equation \eqref{tba}: this should be seen as the natural many-particles generalization of the single-particle, Ooguri-Vafa case. We will not repeat their computations here, but it is instructive to perform a slightly different calculation. Notice that Hitchin (see e.g. \cite{hit2}) spelled out precisely what conditions a set of holomorphic Darboux coordinates must satisfy to give rise to the holomorphic symplectic form of a hyperk\"ahler metric. The crucial point is an integrability condition, requiring that the horizontal derivatives of $\X_{\gamma}(\z)$ must equal the action of suitable vertical complex vector fields. We wish to apply Hitchin's theorem directly to the integral equation. For $\X_e$ we just find
\begin{align*}
\del_a \X_e &= -\frac{1}{\z}i\pi R\del_{\theta_e}\X_e\\
\del_{\bar{a}}\X_e &= - \z i\pi R\del_{\theta_e}\X_e,
\end{align*}
while
\begin{align*}
\del_a \X_m &= \frac{1}{\z} \pi R \tau \X_m + \X_m \left[\frac{i q}{4\pi}\int_{\ell_{\gamma_e}}\frac{d\z'}{\z'}\frac{\z'+\z}{\z'-\z}\del_a\log(1 - \X_e (\z')^q)\right.\\
&\left. -\frac{i q}{4\pi}\int_{\ell_{-\gamma_e}}\frac{d\z'}{\z'}\frac{\z'+\z}{\z'-\z}\del_a\log(1 - \X_e (\z')^{-q})\right]\\
&= \frac{1}{\z} \pi R \tau \X_m + \X_m \left[-\frac{i q^2 R}{4}\int_{\ell_{\gamma_e}}\frac{d\z'}{\z'}\frac{1}{\z'}\frac{\z'+\z}{\z'-\z}\frac{\X_e(\z')^q}{1 - \X_e (\z')^q}\right.\\
&\left. +\frac{i q^2 R}{4}\int_{\ell_{-\gamma_e}}\frac{d\z'}{\z'}\frac{1}{\z'}\frac{\z'+\z}{\z'-\z}\frac{\X_e(\z')^{-q}}{1-\X_e(\z')^{-q}}\right].
\end{align*}
Similarly
\[\del_{\th_m} \X_m = i \X_m,\]
\begin{align*}
-\frac{1}{\z}i\pi R\del_{\th_e} \X_m &= \X_m \left[-\frac{i q^2 R}{4}\int_{\ell_{\gamma_e}}\frac{d\z'}{\z'}\frac{1}{\z}\frac{\z'+\z}{\z'-\z}\frac{\X_e(\z')^q}{1 - \X_e (\z')^q}\right.\\
&\left. + \frac{i q^2 R}{4}\int_{\ell_{-\gamma_e}}\frac{d\z'}{\z'}\frac{\z'+\z}{\z'-\z}\frac{\X_e(\z')^{-q}}{1 - \X_e (\z')^{-q}}\right]\\
&= \X_m \left[-\frac{i q^2 R}{4}\int_{\ell_{\gamma_e}}\frac{d\z'}{\z'}\left(\frac{1}{\z} + \frac{1}{\z'} + \frac{1}{\z'}\frac{\z' + \z}{\z' - \z}\right)\frac{\X_e(\z')^q}{1 - \X_e (\z')^q}\right.\\
&\left. +\frac{i q^2 R}{4}\int_{\ell_{-\gamma_e}}\frac{d\z'}{\z'}\left(\frac{1}{\z} + \frac{1}{\z'} + \frac{1}{\z'}\frac{\z' + \z}{\z' - \z}\right)\frac{\X_e(\z')^{-q}}{1 - \X_e (\z')^{-q}}\right].
\end{align*}
So we find the identity
\begin{equation}\label{CR}
(\del_a + \frac{1}{\z}i\pi R \tau\del_{\theta_m})\X_m =  -\frac{1}{\z}i\pi R\del_{\th_e} \X_m - \left(\frac{1}{\z} v + w \right)\del_{\theta_m}\X_m,
\end{equation}
where $v, w$ are functions defined by
\begin{align*}
v &= -\frac{q^2 R}{4}\int_{\ell_{\gamma_e}}\frac{d\z'}{\z'}\frac{\X_e(\z')^q}{1 - \X_e (\z')^q} + \frac{q^2 R}{4}\int_{\ell_{-\gamma_e}}\frac{d\z'}{\z'}\frac{\X_e(\z')^{-q}}{1 - \X_e (\z')^{-q}},\\ 
w &= -\frac{q^2 R}{4}\int_{\ell_{\gamma_e}}\frac{d\z'}{(\z')^2}\frac{\X_e(\z')^q}{1 - \X_e (\z')^q} + \frac{q^2 R}{4}\int_{\ell_{-\gamma_e}}\frac{d\z'}{(\z')^2}\frac{\X_e(\z')^{-q}}{1 - \X_e (\z')^{-q}}.
\end{align*}
Of course we may rewrite \eqref{CR} as  
\begin{equation*}
\del_a\X_m = \A_a\X_m
\end{equation*}
where the vertical vector field $\A_a$ is given by
\begin{equation}\label{aConnection}
\A_a = \frac{1}{\z}[-i\pi R\del_{\th_e} - (v + i\pi R \tau)\del_{\theta_m}] - w\del_{\th_m}.
\end{equation}
In the notation of \cite{gmn} Section 4, where the principal quantities are the potential $V$ and connection form $A$, we have
\begin{equation*}
v = -\pi V^{\inst},\,\,\,v + i\pi R \tau = -\pi(V + 2\pi i R A_{\th_e}),\,\,\,w = -2\pi A_a,
\end{equation*}
and so 
\begin{equation*}
\A_a = \frac{1}{\z}\left[-i\pi R\del_{\theta_e} + \pi(V + 2\pi i R A_{\theta_e})\del_{\theta_m}\right] + 2\pi A_a \del_{\theta_m}.
\end{equation*}
Similar computations also give the other integrability equation,
\begin{equation*}
\del_{\bar{a}}\X_m = \A_{\bar{a}}\X_m
\end{equation*}
where
\begin{align*}
\A_{\bar{a}} &= - \z\left[i\pi R\del_{\theta_e} - \pi(v - i\pi R \bar{\tau})\del_{\theta_m}\right] - \bar{w}\th_{m}\\
&= 2\pi A_{\bar{a}}\del_{\theta_m} - \z\left[i\pi R\del_{\theta_e} + \pi(V - 2\pi i R A_{\theta_e})\del_{\theta_m}\right].
\end{align*}
\subsection{The GMN connection}  An additional important feature of \cite{gmn} is that the holomorphic Darboux coordinates should be regarded naturally as flat sections of a flat connection on the twistor sphere $\PP^1$. In other words for a fixed choice of parameter $\Lambda, u, R$ we consider the differential equation
\begin{equation*}
\z\del_{\z} \X_{\gamma} = \A_{\z}\X_{\gamma}
\end{equation*}
where $\A_{\z}$ is the complex vertical vector field (on the fibre $\M_{u}$), given in coordinates by
\begin{equation}\label{zetaConnection}
\A_{\z} = \sum_{i,j}\del_{\z}\X_{\gamma_i} [(\del_{\th}\X)^{-1}]_{i j} \del_{\th^j}
\end{equation} 
(where we fix a basis of local sections $\{\gamma_i\}$). There are two claims about $\A_{\z}$:
\begin{enumerate}
\item[$\bullet$] it should decompose as 
\begin{equation*}
\A_{\z} = \frac{1}{\z}\A^{(-1)}_{\z} + \A^{(0)}_{\z} + \z\A^{(1)}_{\z}
\end{equation*}
where the $\A^{(i)}_{\z}$ \emph{do not} depend on $\z$;
\item[$\bullet$] it should satisfy the equation (called the scale invariance/R-symmetry equation)
\begin{equation}\label{scale}
\A_{\z} \X= (- a\del_a + \bar{a}\del_{\bar{a}} - \Lambda\del_{\Lambda} + \bar{\Lambda}\del_{\bar{\Lambda}})\X.
\end{equation}
\end{enumerate}
\textbf{Example.} Again we work this out for the Ooguri-Vafa metric. In fact we show by direct computation that the $\X$ satisfy differential equations of the form
\begin{align*}
\z\del_{\z}\X &= \A_{\z}\X,\\
R\del_R\X &= \A_R\X
\end{align*}
for suitable vertical complex vector fields $\A_{\z}, \A_{R}$ with a very simple $\z$ dependence. An alternative (more enlightning) derivation is given in \cite{gmn} Section 4.5 by exploiting two symmetries of the system.  
\begin{lem}\label{scaleLemma} The following equations hold
\begin{align*}
\z\del_{\z}\X &= \left(-\Lambda\del_{\Lambda} + \bar{\Lambda}\del_{\bar{\Lambda}}-a\del_a + \bar{a}\del_{\bar{a}}\right)\X\\
R\del_R \X &= \left(\Lambda\del_{\Lambda} + \bar{\Lambda}\del_{\bar{\Lambda}} + a\del_a + \bar{a}\del_{\bar{a}}\right)\X.
\end{align*}
\end{lem}
\begin{proof} Consider the $\z \del_{\z}$ equation first. It is straightforward to check it for $\X_e$ (since there are no corrections), and we can compute directly
\begin{align*}
(-a\del_a + \bar{a}\del_{\bar{a}})\X_m &= \left(-\frac{1}{\z} \pi R a\tau + \z \pi R \bar{a}\bar{\tau}\right)\X_m\\
& + \X_m \left[\frac{i q}{4\pi}\int_{\ell_{\gamma_e}}\frac{d\z'}{\z'}\frac{\z'+\z}{\z'-\z}(-a\del_a +\bar{a}\del_{\bar{a}})\log(1 - \X_e (\z')^q)\right.\\
&\left. -\frac{i q}{4\pi}\int_{\ell_{-\gamma_e}}\frac{d\z'}{\z'}\frac{\z'+\z}{\z'-\z}(-a\del_a +\bar{a}\del_{\bar{a}})\log(1 - \X_e (\z')^{-q})\right].
\end{align*}
Now
\[\int_{\ell_{\gamma_e}}\frac{d\z'}{\z'}\frac{\z'+\z}{\z'-\z}(-a\del_a +\bar{a}\del_{\bar{a}})\log(1 - \X_e (\z')^q)\]
\begin{align*}
&= \int_{\ell_{\gamma_e}}\frac{d\z'}{\z'}\frac{\z'+\z}{\z'-\z}\z'\del_{\z'}\log(1 - \X_e (\z')^q)\\
&= -\int_{\ell_{\gamma_e}}d\z'\del_{\z'}\left(\frac{\z'+\z}{\z'-\z}\right)\log(1 - \X_e (\z')^q)\\
&= \z\int_{\ell_{\gamma_e}}\frac{d\z'}{\z'}\frac{2\z'}{(\z'-\z)^2}\log(1 - \X_e (\z')^q)\\
&= \z \del_{\z}\int_{\ell_{\gamma_e}}\frac{d\z'}{\z'}\frac{\z'+\z}{\z'-\z}\log(1 - \X_e (\z')^q), 
\end{align*}
integrating by parts. An identical computation holds for $\int_{\ell_{-\gamma_e}}$. So
\begin{align*}(\z d\z + a\del_a - \bar{a}\del_{\bar{a}})\X_m &= \left(\frac{1}{\z} \pi R (a\tau-Z_m) -\z \pi R (\bar{a}\bar{\tau}-\bar{Z}_m)\right)\X_m\\
&= \left(-\frac{i R q^2 a}{2\z} - \z\frac{iR q^2  \bar{a}}{2}\right)\X_m. 
\end{align*}
(Recall that in the presence of a charge $q > 1$ both $\tau$ and $Z_m$ are rescaled by $q^2$). We conclude thanks to 
\[\Lambda \del_{\Lambda}\X_m = \frac{i R q^2 a}{2\z}\X_m,\,\,\,\bar{\Lambda} \del_{\bar{\Lambda}}\X_m = - \z\frac{iR q^2  \bar{a}}{2}\X_m.\]
The derivation of the $R\del_R$ equation is very similar but simpler (without integration by parts). 
\end{proof}
Let us define further vertical vector fields
\[\A_{\Lambda} = \frac{q^2 R a}{2\z}\del_{\theta_m},\,\,\,\A_{\bar{\Lambda}} = \frac{\z q^2 R \bar{a}}{2}\del_{\theta_m}.\]
\begin{cor} The following equations hold
\begin{align*}
\z\del_{\z}\X &= \A_{\z}\X\\
R \del_R \X &= \A_{R}\X,
\end{align*}
where
\begin{align*}
\A_{\z} &= -a \A_a + \bar{a}\A_{\bar{a}} - \Lambda\A_{\Lambda} + \bar{\Lambda}\A_{\bar{\Lambda}},\\
\A_R &= a\A_a + \bar{a}\A_{\bar{a}} + \Lambda\A_{\Lambda} + \bar{\Lambda}\A_{\bar{\Lambda}}.
\end{align*}
Furthermore there are decompositions
\begin{align*}
\A_{\z} &= \frac{1}{\z}\A^{(-1)}_{\z} + \A^{(0)}_{\z} + \z\A^{(1)}_{\z}\\
\A_R &= \frac{1}{\z}\A^{(-1)}_R + \A^{(0)}_R + \z\A^{(1)}_R
\end{align*}
for vector fields $\A^{i}_{\z}, \A^{j}_{R}$ independent of $\z$, where
\[\A^{(-1)}_R = -\A^{(-1)}_{\z},\,\,\,\A^{(1)}_R = \A^{(1)}_{\z}.\]
\end{cor}
\begin{proof} This is a straightforward computation using the Lemma above and the integrability conditions for $\X$.
\end{proof}
\subsection{The asymptotic expansion}\label{pertSection} In \cite{gmn} Appendix C, the authors apply methods in the analysis of integral equations (especially arguments from \cite{vafa}) to the GMN equation \eqref{tba} in order to find an asymptotic expansion for the solution. Here we only explain briefly the final result. Let $\mathcal{T}$ denote a finite \emph{rooted} tree, with $n$ vertices decorated by elements $\gamma_i \in \Gamma$. Assume for a moment that the vertices are labelled by integers which increase with the distance from the root. The decoration at the root is denoted by $\gamma_{\mathcal{T}}$. \cite{gmn} (C.15) define a weight
\begin{equation}\label{Wweight}
\mathcal{W}_{\mathcal{T}} = (-1)^{n}\frac{f^{\gamma_{\mathcal{T}}}}{|\Aut(\mathcal{T})|}\prod_{(i, j)\in\operatorname{Edges}(\mathcal{T})}\bra \gamma_i, f^{\gamma_j}\ket.
\end{equation}
(There is an extra sign with respect to (C.15) due to a different convention for Kontsevich-Soibelman operators). This weight is clearly intrinsic to $\T$ (it does not depend on the labelling by integers, which we can now forget). Let us denote by $\mathcal{T} \to \{\mathcal{T}_a\}$ the operation of removing the root to produce a finite set of decorated rooted trees. We define the kernel
\begin{equation*}
\rho(\sigma, \tau) = \frac{1}{\tau}\frac{\tau + \sigma}{\tau - \sigma}.
\end{equation*}
Notice that for $f(\sigma)$ holomorphic in a neighborhood of a fixed $\tau_0$ we have 
\begin{equation*}
\operatorname{Res}_{\tau_0} \rho(\sigma, \tau_0) f(\sigma) = 2 f(\tau_0).
\end{equation*} 
\cite{gmn} (C.27) introduce piecewise holomorphic functions (``propagators") $\mathcal{G}_{\mathcal{T}}(\zeta)$, defined recursively by $\mathcal{G}_{\emptyset} = 1$ and
\begin{equation}{\label{Gfunct}}
\mathcal{G}_{\mathcal{T}}(\z) = \frac{1}{4\pi i} \int_{\ell_{\gamma_{\mathcal{T}}}} d\zeta'\rho(\z, \z')\mathcal{X}^{\sf}_{\mathcal{\gamma_{\mathcal{T}}}}(\zeta')\prod_{a} \mathcal{G}_{\mathcal{T}_a}(\zeta').
\end{equation}
Then \cite{gmn} (C.26) claim that a \emph{formal} solution to \eqref{tba} (with the required boundary conditions, which we omit to explain) is given by
\begin{equation}\label{perturbative}
\X_{\gamma}(\z) = \Xsf_{\gamma}(\z) \exp \bra \gamma, \sum_{\mathcal{T}} \mathcal{W}_{\mathcal{T}} \mathcal{G}_{\mathcal{T}}(\z)\ket.
\end{equation}
\begin{rmk} In general the convergence of this asymptotic expansion seems to be an important open problem. From the point of view of Conjecture \ref{mainConj} it is probably related to similar convergence issues in the work of Joyce \cite{joyHolo}.
\end{rmk}
\begin{exm} We work out the asymptotic expansion for the $q$-Ooguri-Vafa. Let us look at the contribution of an edge $(i, j) \subset \mathcal{T}$. This is weighted by $\dt(\gamma_j)$, so $\gamma_j = k q \gamma_e$. But then $\mathcal{W}_{\mathcal{T}}$ vanishes if $\mathcal{T}$ contains a nontrivial edge. We see that only first order instanton corrections survive, parametrised by decorations of a single root $\{\bullet\}$, and we find 
\begin{equation*}
\mathcal{G}_{k q \gamma_e \{\bullet\}} = \frac{1}{4\pi i}\int_{\ell_{\operatorname{sgn}(k)\gamma_e}} d\zeta'\rho(\z, \z')\Xsf_{\gamma_e}(\z')^{k q}
\end{equation*}
(using $\Xsf_{\gamma_1 + \gamma_2} = \Xsf_{\gamma_1} \chi^{\sf}_{\gamma_2}$). By \eqref{perturbative} we have
\begin{align*}
\X_e(\z) &= \Xsf_e(\z),\\
\X_m(\z) &= \Xsf_m(\z)\exp\bra \gamma_m, \sum_{k \in \Z\setminus\{0\}}\mathcal{W}_{k q \gamma_e\{\bullet\}} \mathcal{G}_{k q \gamma_e \{\bullet\}}(\z)\ket\\
&= \Xsf_m(\z)\exp\left[-\frac{q}{4\pi i} \sum_{k \neq 0} \int_{\ell_{\operatorname{sgn}(k)\gamma_e}} d\zeta'\rho(\z, \z')\frac{\Xsf_{\gamma_e}(\z')^{k q}}{k}\right]\\
&= \Xsf_m(\z)\exp\left[\frac{i q}{4\pi}\int_{\ell_{\gamma_e}} d\zeta'\rho(\z, \z')\log(1 - \Xsf_{\gamma_e}(\z')^q)\right.\\
&\left. -\frac{i q}{4\pi}\int_{\ell_{-\gamma_e}} d\zeta'\rho(\z, \z')\log(1 - \Xsf_{\gamma_e}(\z')^{-q})\right].   
\end{align*}
This is the same as the integral equation for the Ooguri-Vafa metric derived in \cite{gmn} (4.33).
\end{exm}
\textbf{Example.} Going back to pure $SU(2)$ Seiberg-Witten, let us analyze the simplest higher order correction to $\log\X_m(\z)$ at strong coupling beyond first order instanton corrections. This is encoded by the graph
\begin{center}
\centerline{
\xymatrix{ \d \ar[r] & \m}}
\end{center}
which corresponds to the integral
\begin{equation*} \frac{2\d}{4\pi i}\int_{\ell_{\d}} d\z_1\rho(\z, \z_1)\Xsf_{\d}(\z_1)\frac{1}{4\pi i}\int_{\ell_{\m}}d\z_2\rho(\z_1, \z_2)\Xsf_{\m}(\z_2).
\end{equation*}
Setting
\begin{equation*}
\z_1 = -\frac{\sqrt{Z_{\d}}}{\sqrt{\bar{Z}_{\d}}} e^{s},\,\,\,\z_2 = -\frac{\sqrt{Z_m}}{\sqrt{\bar{Z}_m}} e^{t},
\end{equation*}
we can estimate this integral by
\begin{align*}
\nonumber C \delta \int^{+\infty}_{-\infty} ds \left|\frac{\z_1 + \z}{\z_1 - \z}\right|&\exp(-2\pi R|Z_{\d}|\cosh(s))\\ &\cdot\int^{+\infty}_{-\infty}dt\left|\frac{\z_2 + \z_1}{\z_2 - \z_1}\right|\exp(-2\pi R|Z_m|\cosh(t)).
\end{align*}
Let $\vartheta$ denote the angle between the BPS rays $\ell_{\d}, \ell_{\m}$. We can estimate the inner integral by
\begin{equation*}
\left(\frac{1 + \cos(\vartheta)}{1 - \cos(\vartheta)}\right)^{1/2} K_0(2\pi R |Z_m|), 
\end{equation*} 
where we used a standard integral representation of the modified Bessel functions of the second kind, 
\begin{equation*}
K_{\alpha}(x) = C_{\alpha}\int^{+\infty}_{-\infty}ds\,e^{-\alpha s} e^{-x\cosh(s)}
\end{equation*}
(for $C_{\alpha}$ a constant which depends only on $\alpha$). Now for $u$ bounded away from $\ms$, the angle $\vartheta$ is bounded away from $0$, and combining this with standard results about the asymptotics of Bessel functions we can estimate the total integral by
\begin{equation*}
C' \delta \frac{1}{\pi R |Z_m|}\exp(-\pi R |Z_m|) \int^{+\infty}_{-\infty} ds \left|\frac{\z_1 + \z}{\z_1 - \z}\right| \exp(-2\pi R|Z_{\d}|\cosh(s)) 
\end{equation*}
for $R$ large enough. Repeating the argument, for fixed $\z$ bounded away from $\ell_{\d}$ we find a uniform bound 
\begin{equation*}
C''\delta \frac{1}{\pi R |Z_{\d}|}\exp(-\pi R |Z_{\d}|)\frac{1}{\pi R |Z_m|}\exp(-\pi R |Z_m|) 
\end{equation*}
as $R\to +\infty$. It is not hard to generalize this example to the following result.
\begin{lem} The contribution of a fixed (rooted, labelled) tree $\T$ with $n$ vertices to the asymptotic expansion \eqref{perturbative} can be estimated by
$C \exp(-C' n R)$ for all $R > C''$, where the constants $C, C', C''$ only depend on the distance of $u$ from the curve $\ms$. 
\end{lem}
A similar result holds for $N_f > 0$ (precisely by the same argument).\\
\textbf{Example} We can apply the asymptotic expansion \eqref{perturbative} to study the GMN connection $\A_{\zeta}$. First notice that the scale invariance/R-symmetry equation \eqref{scale} follows at least formally (modulo convergence issues) from the asymptotic expansion. This is because each ``propagator" $\G_{\T}(\z)$ separately satisfies \eqref{scale}. This is not hard to check by induction, using the recursive definition \eqref{Gfunct} and the fact that \eqref{scale} holds already for the semiflat coordinates $\Xsf_{\gamma}(\z)$. (Indeed we can use this argument and the above asymptotic expansion for the Ooguri-Vafa metric to give a different proof of Lemma \ref{scaleLemma}).

One can also obtain an asymptotic expansion for $\A_{\zeta}$ by combining \eqref{perturbative} with the explicit expression \eqref{zetaConnection}. Picking a basis of local sections $\gamma_e, \m$ as usual, we write the matrix of angular derivatives as 
\begin{equation*}
\del_{\th}\X = \left(\begin{matrix}
\del_{\th_e}\X_{e}     & \del_{\th_e}\X_{m}     \\
\del_{\th_m}\X_{e}    & \del_{\th_m}\X_{m}    
\end{matrix}\right).
\end{equation*}
By \eqref{perturbative}, and using the expression for semiflat coordinates \eqref{sfCoords}, we have
\begin{equation*}
\del_{\theta}\X = \left(\begin{matrix}
i\X_e     & 0 \\
0          & i\X_{m}    
\end{matrix}\right)\left(I - B\right),
\end{equation*}
where $B$ is the matrix of instanton corrections
\begin{equation*}
B = i \left(\begin{matrix}
\bra \gamma_e, \sum_{\T}\W_T\del_{\th_{e}}\G_{\T} \ket   & \bra \gamma_m, \sum_{\T}\W_T\del_{\th_{e}}\G_{\T} \ket\\
\bra \gamma_e, \sum_{\T}\W_T\del_{\th_{m}}\G_{\T} \ket  & \bra \gamma_m, \sum_{\T}\W_T\del_{\th_{m}}\G_{\T} \ket
\end{matrix}\right).
\end{equation*}
According to \eqref{zetaConnection}, the complex vector field $\A_{\z}$ is given in local coordinates by
\begin{align*}
\A_{\zeta}  = \left(\del_{\z}\X_e (\del_{\th}\X)^{-1}_{11} + \del_{\z}\X_m (\del_{\th}\X)^{-1}_{21}\right)\del_{\th_e} + \left(\del_{\z}\X_e (\del_{\th}\X)^{-1}_{12} + \del_{\z}\X_m (\del_{\th}\X)^{-1}_{22}\right)\del_{\th_m}.  
\end{align*}
We have 
\begin{align*}
\del_{\z}\X_{\gamma} = \X_{\gamma}\left(\frac{1}{\z^2} \pi R Z_{\gamma}  + \pi R \bar{Z}_{\gamma} + \bra \gamma, \sum_{\mathcal{T}} \mathcal{W}_{\mathcal{T}} \del_{\z}\mathcal{G}_{\mathcal{T}}\ket \right),
\end{align*}
and combining this with 
\begin{equation*}
(\del_{\th}\X)^{-1} = -i\left(I + B + B^2 + \ldots\right)\left(\begin{matrix}
\X^{-1}_e     & 0 \\
0          & \X^{-1}_{m}    
\end{matrix}\right)
\end{equation*}
we find
\begin{align*}
\A_{\z, \th_e} = &-i \left(\frac{1}{\z^2} \pi R Z_{\gamma_e}  + \pi R \bar{Z}_{\gamma_e}  + \bra \gamma_e, \sum_{\mathcal{T}} \mathcal{W}_{\mathcal{T}} \del_{\z}\mathcal{G}_{\mathcal{T}}\ket \right)  \sum_{k \geq 0} (B^k)_{11}\\
&- i \left(\frac{1}{\z^2} \pi R Z_{\m}  + \pi R \bar{Z}_{\m}  + \bra \m, \sum_{\mathcal{T}} \mathcal{W}_{\mathcal{T}} \del_{\z}\mathcal{G}_{\mathcal{T}}\ket \right)\sum_{k \geq 0} (B^k)_{21},
\end{align*}
and similarly
\begin{align*}
\A_{\z, \th_m} = &-i \left(\frac{1}{\z^2} \pi R Z_{\gamma_e}  + \pi R \bar{Z}_{\gamma_e}  + \bra \gamma_e, \sum_{\mathcal{T}} \mathcal{W}_{\mathcal{T}} \del_{\z}\mathcal{G}_{\mathcal{T}}\ket \right)  \sum_{k \geq 0} (B^k)_{12}\\
&- i \left(\frac{1}{\z^2} \pi R Z_{\m}  + \pi R \bar{Z}_{\m}  + \bra \m, \sum_{\mathcal{T}} \mathcal{W}_{\mathcal{T}} \del_{\z}\mathcal{G}_{\mathcal{T}}\ket \right)\sum_{k \geq 0} (B^k)_{22}.
\end{align*}
In the special case of the Ooguri-Vafa metric the matrix $B$ is nilpotent, 
\begin{equation*}
B = \left(\begin{matrix}
0 & -\frac{1}{4\pi i} \sum_{k \neq 0} \int_{\ell_{\operatorname{sgn}(k)\gamma_e}}d\z'\rho(\z, \z') i \Xsf_{k\gamma_e}(\z')\\
0 & 0
\end{matrix}
\right).  
\end{equation*}
Thus in this case only the magnetic component $\A_{\z, \th_m}$ carries nontrivial instanton corrections, in accordance with the expression for the $a$ component of the GMN connection, \eqref{aConnection}, and the scale invariance/\rm{R}-symmetry equation \eqref{scale}.  

These computations show that \eqref{perturbative} is also an effective tool to study the GMN connection $\z\del_{\z} - \A_{\z}$. As a consequence we believe that relating the asymptotic expansion \eqref{perturbative} to the Joyce-Song formula as in Conjecture \ref{mainConj} could also be helpful in establishing a comparison between the GMN connection and the Bridgeland-Toledano-Laredo connection \cite{bt}.
\subsection{The wall-crossing formula}\label{wallcrossSect} Let $u_0 \in \ms$ denote a smooth point. Moreover assume that the fibre $\M_{u_0}$ is smooth. In particular we can choose local coordinates $(u, \theta)$ around $u_0$ (corresponding to a local trivialization $\Gamma \cong \Z^2$ around $u_0$). In a neighborhood of $u_0$, the complement $\B\setminus \ms$ is the union of two connected components, $U^{\pm}$. We will sometimes write $u^{\pm}$ for a point of $U^{\pm}$. We define
\begin{equation*}
\X_{\gamma}(u^{\pm}_0, \theta; \z) = \lim_{u \in U^{\pm}, u\to u_0}\X_{\gamma}(u, \theta; \z).
\end{equation*}
Notice that both limits exist and are finite, since the central charge $Z(u)$ is well defined around $u_0$. According to \cite{gmn}, the Kontsevich-Soibelman wall-crossing formula can be expressed as the continuity condition
\begin{equation}\label{continuity}
\X_{\gamma}(u^{+}_0, \theta; \z) = \X_{\gamma}(u^{-}_0, \theta; \z)
\end{equation} 
for all $\gamma, \theta$ and a generic, fixed value of $\z$. Since the pairing $\bra - , - \ket$ is nondegenerate, by \eqref{perturbative} this condition is equivalent to
\begin{equation}\label{wallcross}
\sum_{\T} \W^+_{\T} \G^+_{\T}(\zeta) = \sum_{\T} \W^-_{\T} \G^-_{\T}(\zeta),
\end{equation}
where of course  
\begin{equation*}
\G^{\pm} =  \lim_{u \in U^{\pm}, u\to u_0} \G^{\pm},\,\,\, \W^{\pm} = \W(u^{\pm}).  
\end{equation*}
\section{Basic examples from $SU(2)$ Seiberg-Witten theories}\label{basic}
In this section we check Conjecture \ref{mainConj} in a number of examples taken from $0 \leq N_f \leq 3$ Seiberg-Witten theories. In \ref{combiSection} we give a simple graphical calculus to evaluate the contribution to wall-crossing of a GMN diagram, at least for the $N_f = 0$ case, and explain how this turns Conjecture \ref{mainConj} into a purely combinatorial statement. A similar calculus is also available when $N_f > 0$ (although it is slightly more complicated due to the presence of more charges), but rather than explaining this in detail we focus on a few examples that show how to refine the $N_f = 0$ theory. 
\subsection{Standard Seiberg-Witten.} We start with pure $SU(2)$ Seiberg-Witten theory. We illustrate in several cases how the identity \eqref{wallcross} (which arises from the continuity of holomorphic Darboux coordinates, when combined with the asymptotic expansion) induces the Joyce-Song formula \eqref{joyceFormula}. We write $s, w$ for the slope functions induced by the central charge $Z$ at strong and weak coupling respectively, namely for $\gamma \in \Gamma$ a local section
\begin{equation*}
s(\gamma) = \arg Z_{\gamma}(u^+),\,\,\,w(\gamma) = \arg Z_{\gamma}(u^-).
\end{equation*} 
(where $u^{\pm}$ is a point of $U^{\pm}$, a connected component of $\B \setminus \ms$ around $u_0$). Then we have
\begin{equation}\label{slopes}
s(\d) > s(\m),\,\,\,w(\d) < w(\m),
\end{equation}
that is for BPS rays,
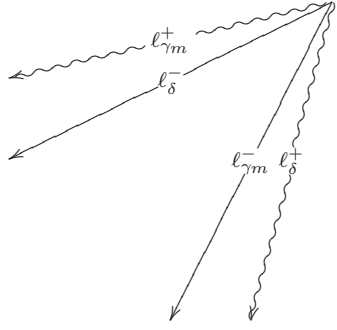
\begin{center}
\centerline{
\xymatrix{   &   &   &   &*=0{}\ar@{~>}[dllll]!U|{\ell_{\m}^+}\ar[ddllll]!U|{\ell_{\d}^-}\ar[ddddll]!U|{\ell_{\m}^-}\ar@{~>}[ddddl]!U|{\ell_{\d}^+}\\
                 &   &   &   &           \\
                 &   &   &   &           \\
                 &   &   &   &           \\
                 &   &   &   &}}
\end{center} 

Throughout this section we assume that the reader is familiar with the Joyce-Song wall-crossing formula, as presented in \cite{js} Section 5. Since there are already excellent short expositions of this formula (including \cite{joySurvey} and \cite{pioline}), we refrain from reviewing it here, but for the reader's convenience we reproduce the explicit formula for the $\U$ functions, equation (3.8) in \cite{js}. Let $\alpha_1, \ldots, \alpha_n$ be a collection of charges (with $n \geq 1$).  If for all $i = 1, \ldots, n-1$ we have either
\begin{enumerate}
\item $s(\alpha_i) < s(\alpha_{i + 1})$ and $w(\alpha_1 + \cdots + \alpha_i) \geq w(\alpha_{i+1} + \cdots + \alpha_n)$, or
\item $s(\alpha_i) \geq s(\alpha_{i + 1})$ and $w(\alpha_1 + \cdots + \alpha_i) < w(\alpha_{i+1} + \cdots + \alpha_n)$
\end{enumerate} 
then one defines $\S(\alpha_1, \ldots, \alpha_n; s, w)$ to be $(-1)^{\#\{\text{indices satisfying (1)}\}}$. Otherwise $\S(\alpha_1, \ldots, \alpha_n; s, w)$ vanishes. Then one defines
\begin{align*}
&\U(\alpha_1,\ldots,\alpha_n; s, w)=\\
&\sum_{\begin{subarray}{l} \phantom{wiggle}\\
1\leq l\leq m\leq n,\;\> 0=a_0<a_1<\cdots<a_m=n,\;\>
0=b_0<b_1<\cdots<b_l=m:\\
\text{Define $\beta_1,\ldots,\beta_m$ by
$\beta_i=\alpha_{a_{i-1}+1}+\cdots+\alpha_{a_i}$.}\\
\text{Define $\gamma_1,\ldots,\gamma_l$ by
$\gamma_i=\beta_{b_{i-1}+1}+\cdots+\beta_{b_i}$.}\\
\text{Then $s(\beta_i)=s(\alpha_j)$, $i=1,\ldots,m$,
$a_{i-1}<j\le a_i$,}\\
\text{and $w(\gamma_i)=w(\alpha_1 + \cdots + \alpha_n)$, $i=1,\ldots,l$}
\end{subarray}
\!\!\!\!\!\!\!\!\!\!\!\!\!\!\!\!\!\!\!\!\!\!\!\!\!\!\!\!\!\!\!\!\!
\!\!\!\!\!\!\!\!\!\!\!\!\!\!\!\!\!\!\!\!\!\!\!\!\!\!\!\!\!\!\!\!\!
\!\!\!\!\!\!\!\!\!\!\!\!\!\!\!\!\!\!\!\!}
\begin{aligned}[t]
\frac{(-1)^{l-1}}{l}\cdot\prod\nolimits_{i=1}^l\S(\beta_{b_{i-1}+1},
\beta_{b_{i-1}+2},\ldots,\beta_{b_i}; s, w)&\\
\cdot\prod_{i=1}^m\frac{1}{(a_i-a_{i-1})!}&\,.
\end{aligned}
\nonumber
\end{align*}
\noindent\textbf{Restriction of root label.} Throughout this section, the roots chosen for the GMN diagrams will always carry a label which is a multiple of $\delta$. It is worth pointing out explicitly that this is because we choose to consider only those GMN diagrams that give a nontrivial correction to $\Xsf_m(\z)$, and for this choice a tree with root carrying a label multiple of $\gamma_m$ would not contribute. We could as well have restricted to diagrams with root label a multiple of $\gamma_m$ (by considering only corrections to $\Xsf_{e}(\z)$).
 
 \subsubsection{$W$ boson of charge $\d + \m$.} The simplest computation concerns $\dt(\d +  \m)$. It has long been known to physicists that in the gauge theory at weak coupling there exists a unique BPS state of charge $\d + \m$, called a ``$W$ boson", which contributes $-2$ to the index $\Omega(\d + \m; u^-)$. This is reflected in the Joyce-Song formula as follows. There is precisely one tree with $2$ vertices labelled by $\{1,2\}$ and a compatible orientation, namely
\begin{center}
\centerline{
\xymatrix{1 \ar[r] & 2}}
\end{center}
The $\U$ symbols of the admissible partitions (which in this case are in fact all the partitions with nonvanishing $\dt$) coincide with the $\S$ symbols,  
\begin{equation*}
\U(\d, \m; s, w) = 1,\,\,\,\U(\m, \d; s, w) = -1. 
\end{equation*}
The first partition contributes
\begin{equation*}
-\frac{1}{2}\U(\d, \m)(-1)^{\bra \d, \m \ket}\bra \d,  \m \ket\dt(\d, s)\dt(\m, s) = -\frac{1}{2} \cdot 1 \cdot 1 \cdot 2 \cdot 1 \cdot 1 = -1,
\end{equation*}
and the second
\begin{equation*}
-\frac{1}{2}\U(\m, \d)(-1)^{\bra \m, \d \ket}\bra \m,  \d \ket\dt(\m, s)\dt(\d, s) = -\frac{1}{2} \cdot (-1) \cdot 1 \cdot (-2) \cdot 1 \cdot 1 = -1,
\end{equation*}
so we find indeed $\dt(\d + \m; s, w) = -2$. Let us consider the analogous decay in GMN theory. We need only consider the rooted, labelled tree (with the induced orientation)
\begin{center}
\centerline{
\xymatrix{*+[F]{\d} \ar[r] & \m}}
\end{center} 
which is present at both strong and weak coupling, and encodes an integral we have already encountered,
\begin{equation*} I(u) = \frac{2\d}{4\pi i}\int_{\ell_{\d}(u)} d\z_1\rho(\z, \z_1)\Xsf_{\d}(u; \z_1)\frac{1}{4\pi i}\int_{\ell_{\m}(u)}d\z_2\rho(\z_1, \z_2)\Xsf_{\m}(u; \z_2).
\end{equation*} 
We proved that $I(u)$ is of order $e^{-2R}$ away from $\ms$. For the wall-crossing however we need to study $I(u^+_0) - I(u^-_0)$, that is the limit of $I(u^+) - I(u^-)$ as $u^{\pm} \to u^0$. We fix $\z$ outside the cone spanned by $\ell_{\d}(u^+), \ell_{\m}(u^-)$. Starting with $I(u^+)$, we can push the first ray of integration $\ell_{\d}(u^+)$ to $\ell_{\d}(u^-)$ without crossing $\ell_{\m}(u^+)$, so we rewrite
\begin{equation*}
I(u^+) = \frac{2\d}{4\pi i}\int_{\ell_{\d}(u^-)} d\z_1\rho(\z, \z_1)\Xsf_{\d}(u^+; \z_1)\frac{1}{4\pi i}\int_{\ell_{\m}(u^+)}d\z_2\rho(\z_1, \z_2)\Xsf_{\m}(u^+; \z_2).
\end{equation*}
The next step is to push the second ray of integration $\ell_{\m}(u^+)$ to $\ell_{\m}(u^-)$. In the process we cross the ray $\ell_{\d}(u^-)$ in the counterclockwise direction, 
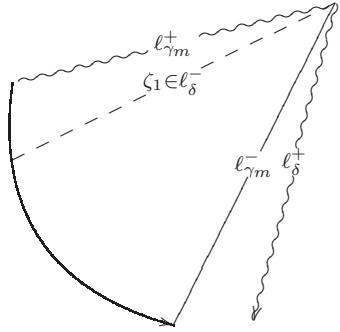
\begin{center}
\centerline{
\xymatrix{   &   &   &   &*=0{}\ar@{~}[dllll]!U|{\ell_{\m}^+}\ar@{--}[ddllll]!U|{\z_1\in\ell_{\d}^-}\ar@{-}[ddddll]!U|{\ell_{\m}^-}\ar@{~>}[ddddl]!U|{\ell_{\d}^+}\\
               *=0{}\ar@/_2pc/@{->}[dddrr] &   &   &   &           \\
                 &   &   &   &           \\
                 &   &   &   &           \\
                 &   &   *=0{}&   &}}
\end{center} 
and so pick up an extra residue of the integrand at $\z_1$, 
\begin{align*}
I(u^+) = \frac{2\d}{4\pi i}&\int_{\ell_{\d}(u^-)} d\z_1\rho(\z, \z_1)\Xsf_{\d}(u^+; \z_1)\\
&\cdot\left(\frac{1}{4\pi i}\int_{\ell_{\m}(u^-)}d\z_2\rho(\z_1, \z_2)\Xsf_{\m}(u^+; \z_2) + \Xsf_{\m}(u^+; \z_1)\right). 
\end{align*}
Since $\Xsf_{\d}(u; \z_1), \Xsf_{\m}(u; \z_2)$ are smooth in a neighborhood of $u_0$, the limit of the first term in $I(u^+)$ as $u^+ \to u_0$ is the same as the limit of $I(u^-)$, therefore
\begin{align*}
I(u^+_0) - I(u^-_0) &= \lim_{u^{\pm}\to 0} \frac{2\d}{4\pi i}\int_{\ell_{\d}(u^-)} d\z_1\rho(\z, \z_1)\Xsf_{\d}(u^+; \z_1)\Xsf_{\m}(u^+; \z_1)\\
&= \lim_{u^{\pm}\to 0} \frac{2\d}{4\pi i}\int_{\ell_{\d}(u^-)} d\z_1\rho(\z, \z_1)\Xsf_{\d + \m}(u^+; \z_1)\\
&= \lim_{u^{\pm}\to 0} \frac{2\d}{4\pi i}\int_{\ell_{\d + \m}(u^-)} d\z_1\rho(\z, \z_1)\Xsf_{\d + \m}(u^+; \z_1)\\
&=  \frac{2\d}{4\pi i}\int_{\ell_{\d + \m}(u_0)} d\z_1\rho(\z, \z_1)\Xsf_{\d + \m}(u_0; \z_1).
\end{align*}
Cancellation requires the existence of the integral at weak coupling
\begin{equation*}
-\frac{2\d}{4\pi i}\int_{\ell_{\d + \m}(u^-)} d\z_1\rho(\z, \z_1)\Xsf_{\d + \m}(u^-; \z_1),
\end{equation*} 
from which we read off $\Omega(\d + \m; u^-) = -2$ as required.\\

\noindent\textbf{Application of the Fubini theorem.} Notice that we may as well have pushed $\ell_{\m}(u^+) \to \ell_{\m}(u^-)$ first, leading to the integral
\begin{equation*}
\frac{2\d}{4\pi i}\int_{\ell_{\d}(u^+)} d\z_1\rho(\z, \z_1)\Xsf_{\d}(u^+; \z_1)\frac{1}{4\pi i}\int_{\ell_{\m}(u^-)}d\z_2\rho(\z_1, \z_2)\Xsf_{\m}(u^+; \z_2).
\end{equation*}
The final result for $\Omega(\d + \m; u^-)$ must of course be the same. To see this notice that in order to change $\ell_{\d}(u^+) \to \ell_{\d}(u^-)$ we need to use the Fubini theorem first, rewriting the integral as
\begin{equation*}
\frac{2\d}{4\pi i}\int_{\ell_{\m}(u^-)}d\z_2\Xsf_{\m}(u^+; \z_2)\frac{1}{4\pi i}\int_{\ell_{\d}(u^+)}d\z_1\rho(\z, \z_1)\rho(\z_1, \z_2)\Xsf_{\d}(u^+; \z_1).
\end{equation*}
Now $\ell_{\d}(u^+) \to \ell_{\d}(u^-)$ crosses $\ell_{\m}(u^-)$ clockwise, so we pick up a $-1$ factor. However this is compensated by the opposite sign of the residue: $\operatorname{Res}_{\z_1 = \z^*_2} \rho(\z_1, \z^*_2) = - \operatorname{Res}_{\z_2 = \z^*_1} \rho(\z^*_1, \z_2)= -2$. So the integral contributes $2\d$, and we find the correct result $\Omega(\d + \m; u^-) = -2$. In more complicated examples it will be necessary to apply the Fubini theorem to reduce the integrals, and one should keep in mind the cancellation of signs pointed out here.\\

\noindent\textbf{Restriction to effective integrals.} We pause for a moment to point out explicitly why, in the computation of $\Omega(a\d + b\m)$ with $a, b\geq 1$, we will only need to consider the contribution of diagrams whose vertices are labelled by \emph{positive} multiples of $\d$ or $\m$. We claim that a diagram $\T$ can only give a contribution to $\Omega(a\d + b\m)$ through wall-crossing if its vertices are all labelled by positive multiples of $\d$ or $\m$. Suppose this is not the case for $\T$, and pick a vertex $v \in \T$ labelled by $-\alpha$, where $\alpha$ is a positive multiple of $\d$ or $\m$, such that $v$ has minimal distance from the root. Thus the integral $\G_{\T}(u^+)$ contains a segment
\begin{equation*}
\cdots \int_{\ell_{-\alpha}(u^+)} d\z'' \rho(\z',\z'')\Xsf_{-\alpha}(u^+,\z'')\cdots 
\end{equation*} 
By our assumption on $v$ and since $\ell_{-\alpha}(u^+) = -\ell_{\alpha}(u^+)$ lies in a half-plane opposite to that of all integration rays that preceed it, we can push $\ell_{-\alpha}(u^+)$ to $\ell_{-\alpha}(u^-)$ without picking up a residue contribution. Furthermore the only integration rays that can cross $\ell_{-\alpha}(u^-)$ (when applying the residue theorem) are again of the form $\ell_{-\beta}(u^+)$, where $\beta$ is a positive multiple of $\d$ or $\m$. It follows that all integrals obtained from $\G_{\T}$ by moving integration rays $\ell(u^+) \to \ell(u^-)$ always contain a factor $\int_{\ell_{-\gamma}(u^-)} d\z'' \rho(\z',\z'')\Xsf_{-\gamma}(u^+,\z'')$, where $\gamma$ is a positive combination of $\d$ and $\m$, and so are never of the correct form to give a contribution to $\Omega(a\d + b\m)$ with $a, b$ positive.
\subsubsection{Dyon of charge $\d + 2\m$.} Classical physical arguments (e.g. \cite{bilal}) predict the existence of a BPS state of charge $\d + 2\m$ at weak coupling, called a dyon, with index $\Omega(\d + 2\m; u^-) = 1$. Let us work out \eqref{joyceFormula} in this case. The $\S$ symbols for partitions are easily derived from \eqref{slopes},
\begin{equation*}
\begin{matrix}
\S(\d, 2\m; s, w) = 1,\,\,\,\S(2\m, \d; s, w) = -1,\\
\S(\d, \m, \m; s, w) = 0,\,\,\,\S(\m, \d, \m; s, w) = -1,\,\,\,\S(\m,\m,\d; s, w) = 1.
\end{matrix}
\end{equation*}
Since the class $\d + 2\m$ is \emph{primitive}, it is easy to derive from this their $\U$ symbols (i.e. in this case they are a weighted sum over contractions),
\begin{equation*}
\begin{matrix}
\U(\d, 2\m; s, w) = 1,\,\,\,\U(2\m, \d; s, w) = -1,\\
\U(\d, \m, \m; s, w) = \frac{1}{2}, \U(\m, \d, \m; s, w) = -1,\,\,\,\U(\m,\m,\d; s, w) = \frac{1}{2}.
\end{matrix}
\end{equation*}
Consider again the tree
\begin{center}
\centerline{
\xymatrix{1 \ar[r] & 2}}
\end{center}
Its compatible ordered partitions are $\d + 2\m$ and $2\m + \d$. The first contributes 
\begin{equation*}
-\frac{1}{2}\U(\d, 2\m)(-1)^{\bra \d, 2\m \ket}\bra \d, 2\m \ket\dt(\d, s)\dt(2\m, s) = -\frac{1}{2} \cdot 1 \cdot 1 \cdot 4 \cdot 1 \cdot \frac{1}{4} = -\frac{1}{2}.
\end{equation*}
Similarly $2\m + \d$ also contributes $-\frac{1}{2}$. The total contribution of the fixed tree is $-1$. On the GMN side we have the integral
\begin{equation*}
\frac{\delta}{4\pi i} \int_{\ell_{\delta}(u)} d\z_1 \rho(\z, \z_1)\chi^{\sf}_{\delta}(u; \z_1) \frac{1}{4\pi i} \int_{\ell_{\gamma_m}(u)} d\z_2 \rho(\z_1, \z_2)\chi^{\sf}_{2\gamma_m}(u; \z_2).
\end{equation*}
By a computation completely analogous to the $\d + \m$ case, cancellation for this integral requires the existence of a term
\begin{equation*}
-\frac{\delta}{4\pi i} \int_{\ell_{\delta + 2\m}(u^-)} d\z_1 \rho(\z, \z_1)\chi^{\sf}_{\delta + 2\m}(\z_1, u^-)
\end{equation*}
in the weak coupling region, which therefore contributes $-1$ to $\Omega(\d + 2\m; u^-)$. Let us go back to the JS side for the remaining trees
\begin{center}
\centerline{
\xymatrix{1 \ar[r] & 2 \ar[r] & 3 &   & 2 & \ar[l] 1 \ar[r] & 3\\
                   &          & 1 \ar[r] & 3 & \ar[l] 2 &   &  }}
\end{center} 
Each of these admits exactly one compatible partition with nonvanishing contribution, e.g. for the first tree this is $\m + \d + \m$, giving 
\begin{equation*}
\frac{(-1)^2}{4} \U(\m, \d, \m; s, w)\bra\m, \d\ket \bra\d, \m\ket\dt(\m)^2\dt(\d) = \frac{1}{4}\cdot(-1)\cdot(-2) \cdot 2 \cdot 1 = 1.
\end{equation*}
The two other trees each contribute $1/2$, so the total contribution here is $2$. This gives the right DT invariant: $\dt(\d + 2\m; u^-) = \Om(\d + 2\m; u^-) = 2 - 1 = 1$. On the GMN side, we need only consider the rooted, labelled tree (with the induced orientation)
\begin{center}
\centerline{
\xymatrix{\m &*+[F]{\d} \ar[l] \ar[r]& \m}
}
\end{center}
This has $\Z/2$ symmetry, so we have $\W = (-1)^3\frac{1}{2}\delta \bra \delta, \m\ket^2 = -2\delta$, and our tree encodes the integral
\begin{align*}
\nonumber I(u^+) &= -\frac{2\d}{4\pi i}\int_{\ell_{\d}(u^+)}d\z_1 \rho(\z, \z_1)\chi^{\sf}_{\d}(\z_1, u^+)\left(\frac{1}{4\pi i}\int_{\ell_{\m}(u^+)}d\z_2\rho(\z_1, \z_2)\chi^{\sf}_{\m}(\z_2, u^+)\right)^2\\
\nonumber &= -\frac{2\d}{4\pi i}\int_{\ell_{\d}(u^-)}d\z_1 \rho(\z, \z_1)\chi^{\sf}_{\d}(\z_1, u^+)\left(\frac{1}{4\pi i}\int_{\ell_{\m}(u^+)}d\z_2\rho(\z_1, \z_2)\chi^{\sf}_{\m}(\z_2, u^+)\right)^2\\
\nonumber &= -\frac{2\d}{4\pi i}\int_{\ell_{\d}(u^-)}d\z_1 \rho(\z, \z_1)\chi^{\sf}_{\d}(\z_1, u^+)\\
&\hskip1cm\cdot\left(\frac{1}{4\pi i}\int_{\ell_{\m}(u^-)}d\z_2\rho(\z_1, \z_2)\chi^{\sf}_{\m}(\z_2, u^+) + \chi^{\sf}_{\m}(\z_1, u^+)\right)^2.
\end{align*}
This splits up as a sum of terms, namely
\begin{equation*}
-\frac{2\d}{4\pi i}\int_{\ell_{\d}(u^-)}d\z_1 \rho(\z, \z_1)\chi^{\sf}_{\d}(\z_1, u^+)\left(\frac{1}{4\pi i}\int_{\ell_{\m}(u^-)}d\z_2\rho(\z_1, \z_2)\chi^{\sf}_{\m}(\z_2, u^+)\right)^2, 
\end{equation*}
\begin{equation*}
-\frac{2\d}{4\pi i}\int_{\ell_{\d}(u^-)}d\z_1 \rho(\z, \z_1)\chi^{\sf}_{\d + \m}(\z_1, u^+)\frac{1}{2\pi i}\int_{\ell_{\m}(u^-)}d\z_2\rho(\z_1, \z_2)\chi^{\sf}_{\m}(\z_2, u^+),
\end{equation*}
and
\begin{equation*}
-\frac{2\d}{4\pi i}\int_{\ell_{\d}(u^-)}d\z_1 \rho(\z, \z_1)\chi^{\sf}_{\d + 2\m}(\z_1, u^+).
\end{equation*}
We only need to take into account the last two terms. The last integral can be rewritten as usual as
\begin{equation*}
-\frac{2\d}{4\pi i}\int_{\ell_{\d + 2\m}(u^-)}d\z_1 \rho(\z, \z_1)\chi^{\sf}_{\d + 2\m}(\z_1, u^+),
\end{equation*}
and therefore its cancellation requires a contribution $2$ to $\Omega(\d + 2\m; u^-)$. On the other hand, we can push the ray $\ell_{\d}(u^-)$ in the second integral to $\ell_{\d + \m}(u^-)$ without crossing $\ell_{\m}(u^-)$. Therefore there is no residue contribution, and no cancellation is required from $\Omega(\d + 2\m; u^-)$. We also make an important observation: the GMN contribution from the tree 
\begin{center}
\centerline{
\xymatrix{\m  &*+[F]{\d}\ar[l]\ar[r] & \m}
}
\end{center}
(i.e. 2) matches the total contribution of the trees
\begin{center}
\centerline{
\xymatrix{\m \ar[r] & \d \ar[r] & \m &   & \m & \ar[l] \d \ar[r] & \m\\
                   &          & \m \ar[r] & \d & \ar[l] \m &   &  }}
\end{center} 
appearing in Joyce-Song (i.e. $1 + \frac{1}{2} + \frac{1}{2}$). Notice that the orientations of Joyce-Song trees are arbitrary, in particular they are not in general induced by the choice of a root.\\
\subsubsection{Dyon of charge $2\d + 3\m$.} At weak coupling, we expect a state (dyon) of charge $2\d + 3\m$, with $\Omega(2\d + 3\m; u^-) = 1$. We will compute with a number of sample trees in GMN theory, and check that the contribution of (all the choices of a root for) a given tree matches the contribution of all its the orientations in JS theory. This is the first computation in which most of the aspects of the full mechanism matching GMN to JS can be seen in action. We start with the unoriented, labelled graph
\begin{center} 
\centerline{
\xymatrix{\d \ar@{-}[r] & \ar@{-}[r] \m \ar@{-}[r] & \d \ar@{-}[r] & 2\m}
}
\end{center}
For GMN theory an orientation is uniquely defined by picking a root. We first analyse the choice
\begin{center} 
\centerline{
\xymatrix{*+[F]{ \d} \ar[r] & \ar[r] \m \ar[r] & \d \ar[r] & 2\m}
}
\end{center}
Since $\W = (-1)^4\d\bra \d, \m\ket \bra \m, \d\ket \bra \d, \frac{1}{2}\m\ket = -4\d$, this diagram encodes an instanton correction at strong coupling given by the integral 
\begin{align*}
\nonumber -2\frac{2\delta}{4\pi i}&\int_{\ell_{\d}(u^+)}d\z_1\rho(\z, \z_1)\chi^{\sf}_{\d}(\z_1, u^+)\frac{1}{4\pi i} \int_{\ell_{\m}(u^+)}d\z_2\rho(\z_1, \z_2)\chi^{\sf}_{\m}(\z_2, u^+)\\
&\cdot\frac{1}{4\pi i} \int_{\ell_{\d}(u^+)}d\z_3\rho(\z_2, \z_3)\chi^{\sf}_{\d}(\z_3, u^+)\frac{1}{4\pi i}\int_{\ell_{\m}(u^+)}d\z_4\rho(\z_3, \z_4)\chi^{\sf}_{2\m}(\z_4, u^+). 
\end{align*}
As usual, we can push the first integration ray $\ell_{\d}(u^+)$ to $\ell_{\d}(u^-)$ with impunity; then pushing the second ray $\ell_{\gamma_m}(u^+)$ to $\ell_{\gamma_m}(u^-)$ splits the integral as 
\begin{align*}
-2\frac{2\delta}{4\pi i}&\int_{\ell_{\d}(u^-)}d\z_1\rho(\z, \z_1)\chi^{\sf}_{\d}(\z_1, u^+)\frac{1}{4\pi i} \int_{\ell_{\m}(u^-)}d\z_2\rho(\z_1, \z_2)\chi^{\sf}_{\m}(\z_2, u^+)\\
&\cdot\frac{1}{4\pi i} \int_{\ell_{\d}(u^+)}d\z_3\rho(\z_2, \z_3)\chi^{\sf}_{\d}(\z_3, u^+)\frac{1}{4\pi i}\int_{\ell_{\m}(u^+)}d\z_4\rho(\z_3, \z_4)\chi^{\sf}_{2\m}(\z_4, u^+)
\end{align*}
plus a residue term
\begin{align*}
-2\frac{2\delta}{4\pi i}&\int_{\ell_{\d}(u^-)}d\z_1\rho(\z, \z_1)\chi^{\sf}_{\d}(\z_1, u^+)\chi^{\sf}_{\m}(\z_1, u^+)\\
&\cdot\frac{1}{4\pi i} \int_{\ell_{\d}(u^+)}d\z_3\rho(\z_1, \z_3)\chi^{\sf}_{\d}(\z_3, u^+)\frac{1}{4\pi i}\int_{\ell_{\m}(u^+)}d\z_4\rho(\z_3, \z_4)\chi^{\sf}_{2\m}(\z_4, u^+). 
\end{align*}
The first integral can only contribute $-2\frac{2\delta}{4\pi i}\int_{\ell_{\d}(u^-)}d\z_1\rho(\z, \z_1)\chi^{\sf}_{\d}(\z_1, u^+) \cdot J$ for some iterated integral $J$, so it does not give top order corrections. On the other hand we may rewrite the residue term as
\begin{align*}
-2\frac{2\delta}{4\pi i}&\int_{\ell_{\d + \m}(u^-)}d\z_1\rho(\z, \z_1)\chi^{\sf}_{\d + \m}(\z_1, u^+)\\
&\cdot\frac{1}{4\pi i} \int_{\ell_{\d}(u^+)}d\z_3\rho(\z_1, \z_3)\chi^{\sf}_{\d}(\z_3, u^+)\frac{1}{4\pi i}\int_{\ell_{\m}(u^+)}d\z_4\rho(\z_3, \z_4)\chi^{\sf}_{2\m}(\z_4, u^+). 
\end{align*}
We iterate the procedure, pushing $\ell_{\d}(u^+)$ to $\ell_{\d}(u^-)$. This gives 
\begin{align}\label{firstIntegral}
\nonumber -2 \frac{2\delta}{4\pi i}&\int_{\ell_{\d + \m}(u^-)}d\z_1\rho(\z, \z_1)\chi^{\sf}_{\d + \m}(\z_1, u^+)\\
&\frac{1}{4\pi i} \int_{\ell_{\d}(u^-)}d\z_3\rho(\z_1, \z_3)\chi^{\sf}_{\d}(\z_3, u^+)\frac{1}{4\pi i}\int_{\ell_{\m}(u^+)}d\z_4\rho(\z_3, \z_4)\chi^{\sf}_{2\m}(\z_4, u^+)
\end{align}
plus a residue term (notice sign change, as $\ell_{\d}(u^+)$ crosses $\ell_{\d + \m}(u^-)$ in the clockwise direction)
\begin{align}\label{someIntegral}
\nonumber 2 \frac{2\delta}{4\pi i}&\int_{\ell_{\d + \m}(u^-)}d\z_1\rho(\z, \z_1)\chi^{\sf}_{\d + \m}(\z_1, u^+)\chi^{\sf}_{\d}(\z_1, u^+)\\
&\cdot\frac{1}{4\pi i}\int_{\ell_{\m}(u^+)}d\z_4\rho(\z_1, \z_4)\chi^{\sf}_{2\m}(\z_4, u^+). 
\end{align}
The main difference is that now the first integral \eqref{firstIntegral} could give a top degree contribution. To see this push the last ray $\ell_{\m}(u^+)$ in \eqref{firstIntegral} to $\ell_{\m}(u^-)$, which splits the integral as
\begin{align*}
\nonumber -2 \frac{2\delta}{4\pi i}&\int_{\ell_{\d + \m}(u^-)}d\z_1\rho(\z, \z_1)\chi^{\sf}_{\d + \m}(\z_1, u^+)\\
&\cdot\frac{1}{4\pi i} \int_{\ell_{\d}(u^-)}d\z_3\rho(\z_1, \z_3)\chi^{\sf}_{\d}(\z_3, u^+)\frac{1}{4\pi i}\int_{\ell_{\m}(u^-)}d\z_4\rho(\z_3, \z_4)\chi^{\sf}_{2\m}(\z_4, u^+)
\end{align*}
plus a residue term
\begin{align}\label{use4Nf2}
\nonumber -2\frac{2\delta }{4\pi i}&\int_{\ell_{\d + \m}(u^-)}d\z_1\rho(\z, \z_1)\chi^{\sf}_{\d + \m}(\z_1, u^+)\\
&\cdot\frac{1}{4\pi i} \int_{\ell_{\d}(u^-)}d\z_3\rho(\z_1, \z_3)\chi^{\sf}_{\d}(\z_3, u^+)\chi^{\sf}_{2\m}(\z_3, u^+).
\end{align}
We need to push the last ray $\ell_{\d}(u^-)$ in the residue term to $\ell_{\d + 2\m}(u^-)$. In doing so however we will cross the integration ray $\ell_{\d + \m}(u^-)$ in the counterclockwise direction:
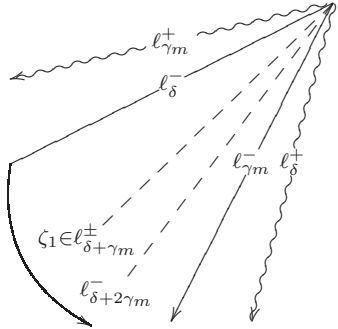
\begin{center}
\centerline{
\xymatrix{   &   &   &   &*=0{}\ar@{~>}[dllll]!U|{\ell_{\m}^+}\ar@{-}[ddllll]!U|{\ell_{\d}^-}\ar@{--}[dddlll]!L|>{\z_1\in\ell_{\d+\m}^{\pm}}\ar@{--}[ddddlll]!U|>>>>>>{\ell_{\d+2\m}^{-}}\ar[ddddll]!U|{\ell_{\m}^-}\ar@{~>}[ddddl]!U|{\ell_{\d}^+}\\
                 &   &   &   &           \\
                 *=0{}\ar@/_1pc/[ddr]&   &   &   &           \\
                 &   &   &   &           \\
                 &   *=0{}&   &   &}}
\end{center}
So we rewrite the whole integral as
\begin{align*}
-2\frac{2\delta}{4\pi i}&\int_{\ell_{\d + \m}(u^-)}d\z_1\rho(\z, \z_1)\chi^{\sf}_{\d + \m}(\z_1, u^+)\\
&\cdot\frac{1}{4\pi i} \int_{\ell_{\d + 2\m}(u^-)}d\z_3\rho(\z_1, \z_3)\chi^{\sf}_{\d + 2\m}(\z_3, u^+)
\end{align*}
plus a residue term
\begin{align}\label{top1}
\nonumber -2\frac{2\delta}{4\pi i}&\int_{\ell_{\d + \m}(u^-)}d\z_1\rho(\z, \z_1)\chi^{\sf}_{\d + \m}(\z_1, u^+)\chi^{\sf}_{\d + 2\m}(\z_1, u^+)\\
&= -2\frac{2\delta}{4\pi i}\int_{\ell_{2\d + 3\m}(u^-)}d\z_1\rho(\z, \z_1)\chi^{\sf}_{2\d + 3\m}(\z_1, u^+).
\end{align}
Thus we see the first top degree integral appear. Its cancellation requires a contribution of $2$ to $\Omega(2\d + 3\m; u^-)$. Going back to the integral \eqref{someIntegral}, this can be rewritten as 
\begin{align*} 
2\frac{2\delta}{4\pi i}&\int_{\ell_{2\d + \m}(u^-)}d\z_1\rho(\z, \z_1)\chi^{\sf}_{2\d + \m}(\z_1, u^+)\\
\cdot&\frac{1}{4\pi i}\int_{\ell_{\m}(u^-)}d\z_4\rho(\z_1, \z_4)\chi^{\sf}_{2\m}(\z_4, u^+)
\end{align*}
plus the residue term
\begin{align}\label{top2}
\nonumber 2 &\frac{2\delta}{4\pi i}\int_{\ell_{2\d + \m}(u^-)}d\z_1\rho(\z, \z_1)\chi^{\sf}_{2\d + \m}(\z_1, u^+)\chi^{\sf}_{2\m}(\z_1, u^+)\\
&= 2 \frac{2\delta}{4\pi i}\int_{\ell_{2\d + 3\m}(u^-)}d\z_1\rho(\z, \z_1)\chi^{\sf}_{2\d + 3\m}(\z_1, u^+). 
\end{align}
This is the only other top degree integral arising from the present diagram. Summing up the two top degree contributions \eqref{top1} and \eqref{top2} then we find that the present diagram gives no contribution to $\Om(2\d + 3\m; u^-)$: we get total contribution $+2 - 2 =0$. A very similar analysis can be performed on the GMN diagram obtained from the other possible choice of a root,
\begin{center}
\centerline{
\xymatrix{\d & \ar[l] m &*+[F]{\d}\ar[l]\ar[r] & 2\m}
}
\end{center}
This again shows that the diagram gives vanishing contribution to $\Om(2\d + 3\m; u^-)$. Let us compare this to the situation in Joyce-Song theory. For this we need to sum over all partitions and $\{1,2,3,4\}$-labelled trees which yield the same unoriented $\Gamma$-labelled tree; this lengthy calculation can be summarized as 
\begin{center}
\centerline{
\xymatrix{\d \ar[r]& \m & \ar[l]\d & \ar[l]2\m\,\,\sim -\frac{1}{2},\,\,\,\,\,\,\,\,\,\d \ar[r]& \m & \ar[l]\d\ar[r] & 2\m\,\,\sim \frac{1}{2}}}
\end{center}
while the contribution of each of the other possible orientations vanishes. We verify once again that the sum over all JS diagrams with the same underlying $\Gamma$-labelled tree matches the same quantity in GMN theory, although the weight of each single \emph{orientation} is very different in the two theories (i.e. in the present example, cancellation happens in a very different way). The same happens with the other distinguished $\Gamma$-labelling. Indeed, we can check that each of the GMN diagrams   
\begin{center}
\centerline{
\xymatrix{*+[F]{\d} \ar[r] & 2\m \ar[r] & \d \ar[r]& \m & & \d & \ar[l] 2\m &*+[F]{\d}\ar[l]\ar[r] & \m}
}
\end{center}
requires a contribution of $-4$ units to the index $\Om(2\d +3\m; u^-)$ for its cancellation. Similarly, the sum over all possible partitions and orientations in JS theory equals $-8$, although in a rather different way: one can show that each single orientation of the diagram 
\begin{center}
\centerline{
\xymatrix{\d \ar@{-}[r]& 2\m \ar@{-}[r] & \d \ar@{-}[r] & \m}}
\end{center}
gives the same contribution (i.e. $-1$) to $\Om(2\d +3\m; u^-)$. In the rest of this subsection we concentrate on the diagram
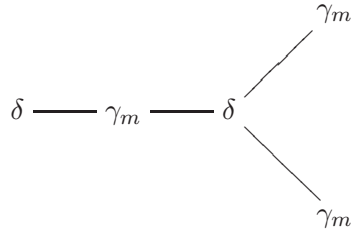
\begin{center}
\centerline{
\xymatrix{                    &                       &                                                & \m \\
              \d \ar@{-}[r]  & \m \ar@{-}[r]   &   \d\ar@{-}[dr]\ar@{-}[ur]                        &       \\
                                  &                       &                                                & \m }
}
\end{center}
In GMN theory we can frame it in two ways, both with $\Z/2$ symmetry, 
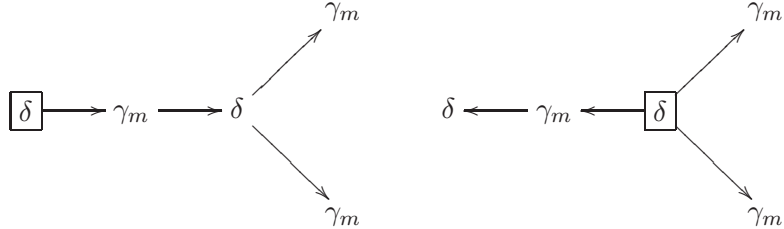
\begin{center}
\centerline{
\xymatrix{                    &                       &                                                & \m   &       &            &                     &  \m   &     \\
              *+[F]{\d} \ar[r]         & \m \ar[r]          &   \d\ar[dr]\ar[ur]              &       & \d   &  \ar[l]\m& *+[F]{\d}\ar[l]\ar[ur]\ar[dr]&        &     \\
                                  &                       &                                                & \m   &       &            &                     &  \m  &}
}
\end{center}
For the first choice we have $\W = (-1)^5\frac{1}{2}\d\bra \d, \m\ket \bra \m, \d\ket \bra \d, \m\ket^2 = 8\d$, and the diagram gives an instanton correction, at strong coupling,
\begin{align*}
4\frac{2\delta}{4\pi i}&\int_{\ell_{\d}(u^+)}d\z_1\rho(\z, \z_1)\chi^{\sf}_{\d}(\z_1, u^+)\frac{1}{4\pi i} \int_{\ell_{\m}(u^+)}d\z_2\rho(\z_1, \z_2)\chi^{\sf}_{\m}(\z_2, u^+)\\
&\cdot\frac{1}{4\pi i} \int_{\ell_{\d}(u^+)}d\z_3\rho(\z_2, \z_3)\chi^{\sf}_{\d}(\z_3, u^+)\left(\frac{1}{4\pi i}\int_{\ell_{\m}(u^+)}d\z_4\rho(\z_3, \z_4)\chi^{\sf}_{2\m}(\z_4, u^+)\right)^2. 
\end{align*}
As usual pushing $\ell_{\d}(u^+)$ to $\ell_{\d}(u^-)$, then $\ell_{\m}(u^+)$ to $\ell_{\m}(u^-)$ splits the integral as 
\begin{align*}
4\frac{2\delta}{4\pi i}&\int_{\ell_{\d}(u^-)}d\z_1\rho(\z, \z_1)\chi^{\sf}_{\d}(\z_1, u^+)\frac{1}{4\pi i} \int_{\ell_{\m}(u^-)}d\z_2\rho(\z_1, \z_2)\chi^{\sf}_{\m}(\z_2, u^+)\\
&\cdot\frac{1}{4\pi i} \int_{\ell_{\d}(u^+)}d\z_3\rho(\z_2, \z_3)\chi^{\sf}_{\d}(\z_3, u^+)\left(\frac{1}{4\pi i}\int_{\ell_{\m}(u^+)}d\z_4\rho(\z_3, \z_4)\chi^{\sf}_{2\m}(\z_4, u^+)\right)^2. 
\end{align*}
plus a residue 
\begin{align*}
4\frac{2\delta}{4\pi i}&\int_{\ell_{\d}(u^-)}d\z_1\rho(\z, \z_1)\chi^{\sf}_{\d + \m}(\z_1, u^+)\\
&\cdot\frac{1}{4\pi i} \int_{\ell_{\d}(u^+)}d\z_3\rho(\z_1, \z_3)\chi^{\sf}_{\d}(\z_3, u^+)\left(\frac{1}{4\pi i}\int_{\ell_{\m}(u^+)}d\z_4\rho(\z_3, \z_4)\chi^{\sf}_{2\m}(\z_4, u^+)\right)^2. 
\end{align*}
It is easy to check that the first integral dies out: pushing rays around with the residue theorem will never produce a top degree correction. On the other hand, by pushing $\ell_{\d}(u^-)$ to $\ell_{\d + \m}(u^-)$, then $\ell_{\d}(u^+)$ to $\ell_{\d}(u^-)$, the residue decays to the integrals
\begin{align*}
4\frac{2\delta}{4\pi i}&\int_{\ell_{\d + \m}(u^-)}d\z_1\rho(\z, \z_1)\chi^{\sf}_{\d + \m}(\z_1, u^+)\\
&\cdot\frac{1}{4\pi i} \int_{\ell_{\d}(u^-)}d\z_3\rho(\z_1, \z_3)\chi^{\sf}_{\d}(\z_3, u^+)\left(\frac{1}{4\pi i}\int_{\ell_{\m}(u^+)}d\z_4\rho(\z_3, \z_4)\chi^{\sf}_{2\m}(\z_4, u^+)\right)^2 
\end{align*}
and
\begin{equation*}
-4\frac{2\delta}{4\pi i}\int_{\ell_{\d + \m}(u^-)}d\z_1\rho(\z, \z_1)\chi^{\sf}_{2\d + \m}(\z_1, u^+)\left(\frac{1}{4\pi i}\int_{\ell_{\m}(u^+)}d\z_4\rho(\z_1, \z_4)\chi^{\sf}_{2\m}(\z_4, u^+)\right)^2.
\end{equation*}
Pushing $\ell_{\m}(u^+)$ to $\ell_{\m}(u^-)$ in the first integral gives a residue
\begin{equation*}
4\frac{2\delta}{4\pi i}\int_{\ell_{\d + \m}(u^-)}d\z_1\rho(\z, \z_1)\chi^{\sf}_{\d + \m}(\z_1, u^+)\frac{1}{4\pi i} \int_{\ell_{\d}(u^-)}d\z_3\rho(\z_1, \z_3)\chi^{\sf}_{\d + 2\m}(\z_3, u^+),
\end{equation*}
and finally pushing $\ell_{\d}(u^-)$ to $\ell_{\d + 2\m}(u^-)$ crosses $\ell_{\d + \m}(u^-)$ in the counterclockwise direction, giving a top order contribution of $-4$ to $\Omega(u^-)$. It is even easier to check that the second integral contributes instead $+4$ to $\Omega(u^-)$, proving that the total contribution of the present framed diagram vanishes. Similar computations show that the other choice of framing also gives a vanishing contribution. In Joyce-Song theory we get the same vanishing, but in a very different way: indeed the only oriented diagrams which carry a JS contribution are
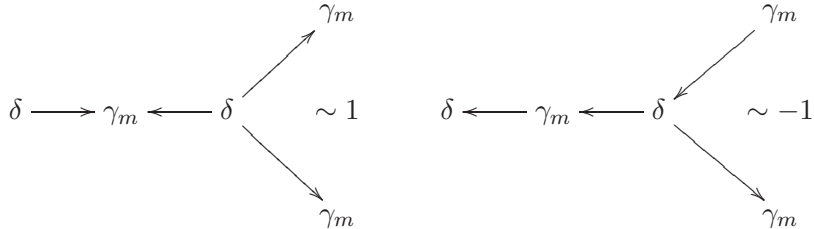
\begin{center}
\centerline{
\xymatrix{                    &                       &                                                & \m   &       &            &                     &  \ar[dl]\m   &     \\
              \d \ar[r]         & \m           &   \ar[l]\d\ar[dr]\ar[ur]                        &\sim 1       & \d   &  \ar[l]\m&    \d\ar[l]\ar[dr]& \sim -1       &     \\
                                  &                       &                                                & \m   &       &            &                     &  \m  &}
}
\end{center}
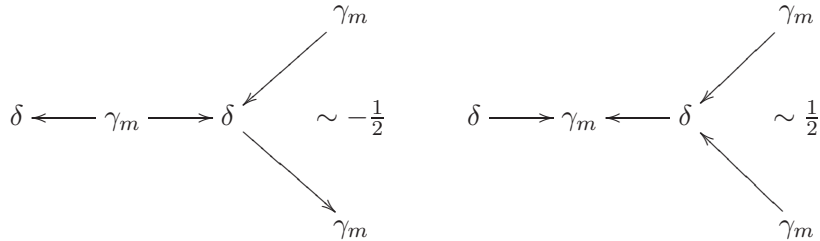
\begin{center}
\centerline{
\xymatrix{                    &                       &                                                & \ar[dl]\m   &       &            &                     &  \ar[dl]\m   &     \\
              \d         &\ar[l] \m \ar[r]          &   \d\ar[dr]              &  \sim -\frac{1}{2}     & \d \ar[r]  &  \m&   \d\ar[l]&  \sim \frac{1}{2}      &     \\
                                  &                       &                                                & \m   &       &            &                     &  \ar[ul]\m  &}
}
\end{center}
\subsection{A purely combinatorial formulation}\label{combiSection} As the reader probably guessed, it is possible to make the above computations with GMN diagrams completely systematic, giving a graphical procedure to evaluate the contribution of each diagram. This turns Conjecture \ref{mainConj} for $N_f = 0$ Seiberg-Witten into a purely combinatorial statement, which nevertheless we do not know how to prove at the moment. There are a number of open questions with a similar combinatorial flavour which seem very relevant to wall-crossing theory, see for example the conjectures of Manschot, Pioline and Sen in \cite{mps}.

We now describe a process that computes the contribution of a GMN diagram by a finite sequence of decays into shorter diagrams. Iteratively we denote by $\T$ one of the diagrams produced in the process. Initially $\T$ is a GMN diagram $\T^0$ at strong coupling. Its vertices are labelled by classes $\gamma_i \in \Gamma$, which are in one to one correspondence with BPS integration rays  $\ell_{\gamma_i}(u^+)$ in the underlying iterated integral $\W_{T^0} \G_{T^0}$. So for simplicity of notation, at the initial step of the process, we think of the vertices as labelled by $\gamma^+_i$. Consider first the following operation:
\begin{enumerate}
\item[$\bullet$] A vertex $\gamma^+_i$ of $\T$ which has minimum distance to $\gamma_{\T}$ transforms to $\gamma^-_i$.
\end{enumerate}
This represents graphically the operation of pushing the corresponding BPS integration ray from $\ell_{\gamma_i}(u^+)$ to $\ell_{\gamma_i}(u^-)$. At the very first step $\gamma^+_i$ is just the root $\gamma^+_{\T^0}$, and as we have seen we can replace this with $\gamma^-_{\T^0}$ freely. At a general step we are focusing on a subtree of $\T$ of the form 
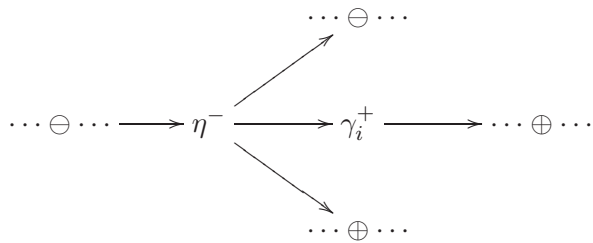
\begin{center}
\centerline{
\xymatrix{
                                     &                            &     \cdots \ominus \cdots                      &        \\
\cdots\ominus \cdots \ar[r] & \eta^- \ar[r]\ar[ur]\ar[dr] & \gamma^+_i \ar[r] & \cdots \oplus \cdots\\
                                     &                            &     \cdots \oplus \cdots                      &}}
\end{center}
where we have denoted by $\ominus$ ($\oplus$) a collection of charges of the form $\xi^-$ (respectively $\xi^+$). The charge $\gamma^+_i$ trasforms to $\gamma^-_i$, giving a new diagram $\T^-$ where the above segment is just replaced by 
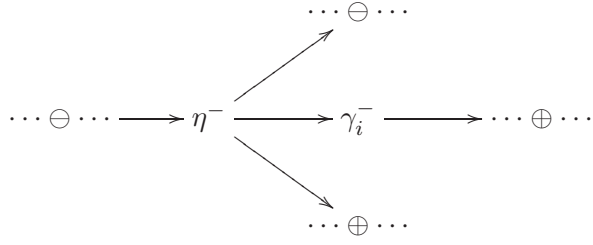
\begin{center}
\centerline{
\xymatrix{
                                     &                            &     \cdots \ominus \cdots                      &        \\
\cdots\ominus \cdots \ar[r] & \eta^- \ar[r]\ar[ur]\ar[dr] & \gamma^-_i \ar[r] & \cdots \oplus \cdots\\
                                     &                            &     \cdots \oplus \cdots                      &}}
\end{center}
However transforming $\gamma^+_i$ to $\gamma^-_i$ entails pushing $\ell_{\gamma_i}(u^+)$ to $\ell_{\gamma_i}(u^-)$. In doing so we may happen to cross the BPS ray $\ell_{\eta}(u^-)$. In this case we say that $\eta^-$  and $\gamma^+_i$ interact.\\ 
\begin{rmk} Notice that in performing this operation $\gamma^+_i$ can never interact with vertices which lie farther from the root (i.e. one of the cloud of $\oplus$ to its right), because these are still labelled by original BPS charges $\gamma^+_j$, which lie at the boundary of the cone spanned by $\ell_{\d}(u^+), \ell_{\m}(u^+)$, while interaction (that is, crossing over of BPS rays) can only happen in the interior cone spanned by $\ell_{\d}(u^-), \ell_{\m}(u^-)$.
\end{rmk}

\noindent The residue theorem shows that there is a further decay product, a diagram $\T_{res}$ where the original segment is replaced by
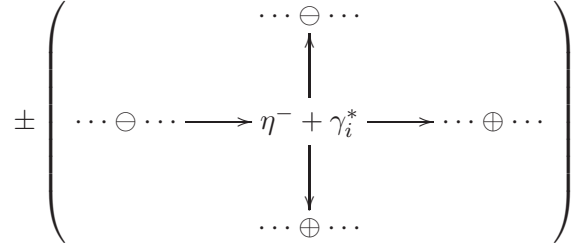
\begin{equation*}
\pm\left(\vcenter{\xymatrix{    &                         \cdots \ominus \cdots                                &                            \\
\cdots\ominus \cdots \ar[r] & \eta^- + \gamma^*_i  \ar[r]\ar[u]\ar[d]                             & \cdots \oplus \cdots\\
                                     &                                 \cdots \oplus \cdots                      &}}\right)
\end{equation*}
We pick the sign $\pm$ according to whether $\ell_{\gamma_i}(u^-)$ crosses $\ell_{\eta}(u^-)$ in the counterclockwise, respectively clockwise direction. We see a new crucial piece of notation appearing here: a vertex of the form $\eta^- + \xi^*$ corresponds to an ``unbalanced" integral
\begin{equation*}
\cdots \frac{1}{4\pi i}\int_{\ell_{\eta}(u^-)}d\tau\rho(\sigma, \tau)\Xsf_{\eta + \xi}(\tau, u^+) \cdots
\end{equation*} 
i.e. one in which the BPS integration ray $\ell_{\eta}(u^-)$ disagrees with the charge of the (piece of) integrand $\Xsf_{\eta + \xi}(\tau, u^+)$. So we come to the second operation: 
\begin{enumerate}
\item[$\bullet$] A vertex of the form $\eta^- + \xi^*$ transforms to $(\eta + \xi)^-$.
\end{enumerate}
This represents graphically the operation of pushing the BPS integration ray from $\ell_{\eta}(u^-)$ to $\ell_{\eta + \xi}(u^-)$. At a general step we are focusing on a subtree of $\T$ of the form 
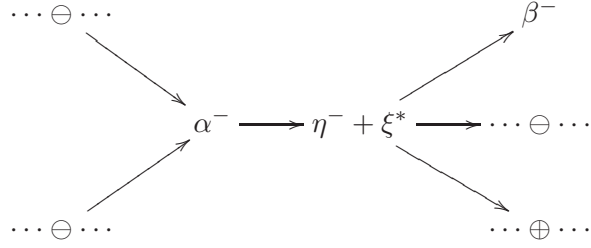
\begin{center}
\centerline{
\xymatrix{
\cdots\ominus \cdots \ar[dr]                                    &                            &                              &   \beta^{-}   \\
                                                                & \alpha^- \ar[r]           &  \eta^- + \xi^* \ar[ur]\ar[dr]\ar[r]  & \cdots \ominus \cdots                               \\
\cdots\ominus \cdots \ar[ur]                                     &                            &                               &   \cdots \oplus \cdots  }}
\end{center}
The vertex $\eta^- + \xi^*$ transforms to $(\eta + \xi)^-$, replacing $\T$ with the diagram $\T^-$ given by
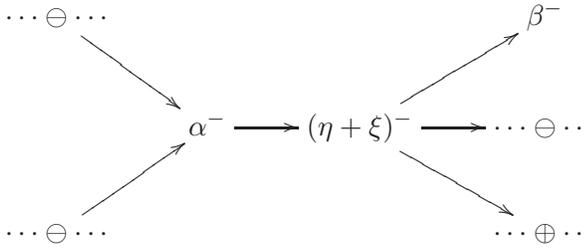
\begin{center}
\centerline{
\xymatrix{
\cdots\ominus \cdots \ar[dr]                                    &                            &                              &    \beta^-  \\
                                                                & \alpha^- \ar[r]           &  (\eta + \xi)^- \ar[ur]\ar[dr]\ar[r]  & \cdots \ominus \cdots                               \\
\cdots\ominus \cdots \ar[ur]                                     &                            &                               &   \cdots \oplus \cdots  }}
\end{center}
But while pushing $\ell_{\eta}(u^-)$ to $\ell_{\eta + \xi}(u^-)$ we may happen to cross one or both of the integration rays $\ell_{\alpha}(u^-)$ and $\ell_{\beta}(u^-)$ (this is of course a schematic picture; in general we may cross more integration rays, both incoming and outgoing at $\eta^- + \xi^*$). If so applying the Fubini and residue theorems shows that $\eta^- + \xi^*$ and $\alpha^-$ (and possibly, $\beta^-$) interact, producing a residue diagram $\T_{res}$
\begin{equation*}
\pm\left(\vcenter{\xymatrix{
\cdots\ominus \cdots \ar[dr]                                    &                                                          &   \beta^-   \\
                                                                          & \alpha^- + (\eta + \xi)^* \ar[ur]\ar[dr]\ar[r]       & \cdots \ominus \cdots                                \\
\cdots\ominus \cdots \ar[ur]                                    &                                                          &   \cdots \oplus \cdots  }}\right)
\end{equation*}
and possibly
\begin{equation*}
\pm\left(\vcenter{\xymatrix{
\cdots\ominus \cdots \ar[dr]                                    &                                                          &    \beta^- + (\eta + \xi)^*  \\
                                                                          & \alpha^-  \ar[ur]\ar[dr]\ar[r]       &    \cdots \ominus \cdots                            \\
\cdots\ominus \cdots \ar[ur]                                    &                                                          &   \cdots \oplus \cdots  }}\right)
\end{equation*}
\noindent (again, we just illutrate the situation with two charges $\alpha^-, \beta^-$; there may be more entirely similar diagrams, coming from incoming and outgoing edges at $\eta^- + \xi^*$). We apply recursively the two operations described above to the initial GMN diagram at strong coupling. At each step of the process we will have in general many decays of diagrams, each of the form $\T \to \T^- \pm \T_{res}$. After a finite number of steps we are left with a finite set of \emph{signed} diagrams. The signed diagrams with more than just a vertex (i.e. those which are not singletons) encode higher order corrections at weak coupling, which we may ignore. We are only interested in the finite set of singleton diagrams $\{\eps_i \T_i\}$. The total contribution of the original GMN diagram $\T^0$ to $\Omega(p \delta + q \m; u^-)$ is then given by
\begin{equation*}
\Omega^-_{\T^0} = -\frac{1}{p}\W_{\T^0}\sum_{i}\eps_i. 
\end{equation*}
As a nontrivial example let us check again using this procedure that the contribution of the following diagram vanishes:
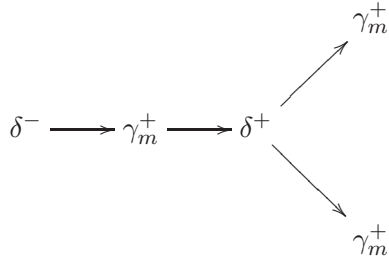
\begin{center}
\centerline{
\xymatrix{                    &                       &                                                & \m^+   &\\
              \d^-\ar[r]         & \m^+ \ar[r]          &   \d^+\ar[dr]\ar[ur]  &            &\\
                                  &                       &                                                & \m^+   &}
}
\end{center}
The first decay is 
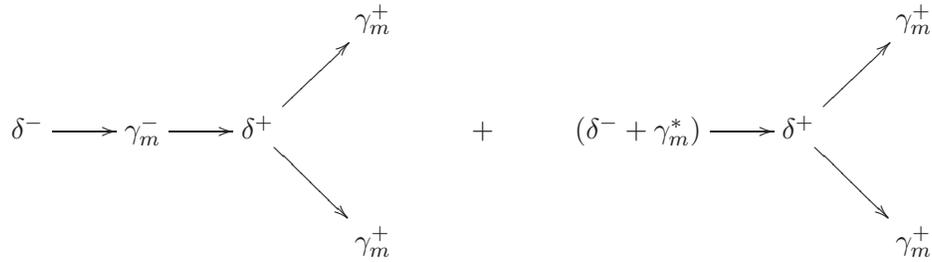
\begin{center}
\centerline{
\xymatrix{                    &                       &                                                & \m^+&   &       &                             &  \m^+   &     \\
              \d^- \ar[r]      & \m^- \ar[r]          &   \d^+\ar[dr]\ar[ur]              &       & + &(\d^- + \m^*) \ar[r]  &  \d^+\ar[ur]\ar[dr]&        &     \\
                                  &                       &                                                & \m^+ &  &       &                             &  \m^+  &}
}
\end{center}
The first diagram dies out, while the latter further decays to
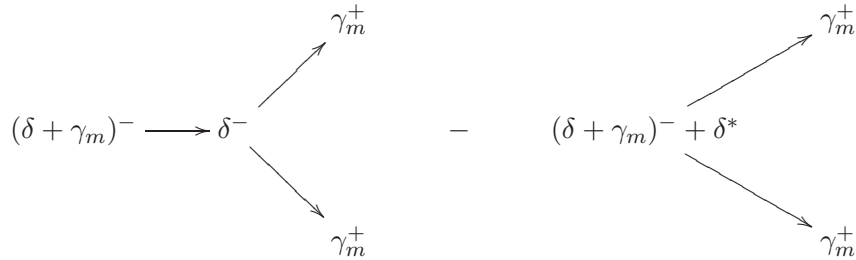
\begin{equation*}
\xymatrix{                                 &                                     &  \m^+  &    &                    &  \m^+ \\
                 (\d + \m)^- \ar[r]     &  \d^-\ar[ur]\ar[dr]           &            & -  &  (\d + \m)^- + \d^* \ar[ur]\ar[dr]  \\
                                               &                                     &  \m^+  &    &                    &  \m^+ }\,\,\,\,\,\,\,\,\,\,\,\,\,\,\,\,\,\,\,\,\,
\end{equation*}
In turn the first of these diagrams decays to 
\begin{equation*}
\xymatrix{
                                               &                                     &  \m^-  &    &                     &\\
                 (\d + \m)^- \ar[r]     &  \d^-\ar[ur]\ar[dr]           &           & +  &  (\d + \m)^- \ar[r] & (\d^- + 2\m^*)\,\,\,\,\,\,\,\,\,\,\,\,\\
                                               &                                     &  \m^-  &    &                     &}
\end{equation*}
which picks up a top degree contribution from the last term, 
\begin{equation*}
\xymatrix{
(\d + \m)^- \ar[r] & (\d^- + 2\m)^- & + & (2\d + 3\m)^-\,\,\,\,\,\,\,\,\,\,\,\,\,\,\,\,\,\,\,\,\,\,\,\,\,\,\,\,\,\,\,\,\,\,\,\,\,\,\,\,\,\,\,\,\,\,\,\,\,\,\,\,\,\,\,
}
\end{equation*}
On the other hand we have the decay into 
\begin{equation*}
\xymatrix{                          &  \m^- & \\
               - (2\d + \m)^- \ar[ur]\ar[dr] & & - & (2\d + 3\m)^-&\,\,\,\,\,\,\,\,\,\,\,\,\,\,\,\,\,\,\,\,\,\\
                                       &  \m^- &}
\end{equation*}
So altogether in this case we have $\sum_i \eps_i = +1 - 1 = 0$.\\
\noindent{\textbf{Remark.}} We emphasize that when we perform the operation $\eta^- + \xi^* \to (\eta + \xi)^-$ at a vertex, interactions will in general occur with neighbouring vertices which may be both closer and farther from the root. As a simple example the diagram\begin{center}
\centerline{
\xymatrix{              &                  &                       &                                                & \d^+ \ar[r]  & \m^+ &\\
              \d^-\ar[r]  &  \m^+\ar[r]      & 2\d^+ \ar[r]          &   \m^+\ar[dr]\ar[ur]  &            &\\
                               &                &                       &                                                & \d^+  \ar[r] & 3\m^+ &}
}
\end{center}
contains among its decay diagrams 
\begin{center}
\centerline{ 
\xymatrix{                  &                                    &        (\d+\m)^-         &\\
(\d + \m)^- \ar[r]         & (2\d + \m)^-  \ar[ur]\ar[dr]   &                  &\\
                                   &                                    &        \d^- + 3\m^*     &}}
\end{center}
and so
\begin{center}
\centerline{ 
\xymatrix{ 
(\d + \m)^- \ar[r]         & (2\d + \m)^- + (\d + 3\m)^*  \ar[r]   &  (\d+\m)^-             &}}
\end{center}
Using the Fubini and residue theorems we see that the middle term interacts with both neighbours.
\subsection{$N_f > 0$ features} There are two features of the GMN setup which we have ignored so far, but which become relevant when $N_f > 0$: firstly, the local system $\hat{\Gamma}$ becomes larger than $\Gamma$; and secondly, the form $\bra-, -\ket$ on the sublattice of $\hat{\Gamma}$ spanned by BPS charges is no longer even. According to \cite{gmn} the definitions of the semiflat Darboux coordinates need to be modified, including a choice of \emph{quadratic refinement} $\sigma$ on $\hat{\Gamma}$. This is a locally defined function which satisfies $\sigma(\gamma_1 + \gamma_2) = (-1)^{\bra \gamma_1, \gamma_2\ket}\sigma(\gamma_1)\sigma(\gamma_2)$. The essential point for us is that the definition \eqref{fCoeff} of $f^\gamma$ must be modified to
\begin{equation*}
f^{\gamma} = \sum_{n > 0, \gamma = n\gamma'} \frac{\sigma(\gamma')^n}{n}\Omega(\gamma'; u)\gamma',
\end{equation*}  
We need to keep track of this in the definition \eqref{Wweight} of the weight $\W_{\T}$. Notice that in the $N_f > 0$ case we have $f^{\gamma}$ differs from $\Omega(\gamma)$ even for primitive classes (by the factor $\sigma(\gamma)$). The quadratic refinement $\sigma$ for GMN diagrams in the analogue of the total sign $\prod(-1)^{\bra \alpha_i, \alpha_j\ket}$ in Joyce-Song theory. In practice, as we will see, this means that to get the right answer from JS computations for a diagram $\T'$ labelled by $\alpha_1, \ldots, \alpha_n$ we must use the usual $\dt$ invariants in JS (i.e. untwisted by $\sigma$), and then multiply by the factor $\left(\prod (-1)^{\bra \alpha_i, \alpha_j\ket}\right)^{-1}\left(\prod \sigma(\alpha_k)\right)$.

Finally, since the form $\bra -, -\ket$ is degenerate on $\hat{\Gamma}$, the equality \eqref{wallcross} is no longer a consequence of the continuity condition \eqref{continuity}, so we \emph{assume} \eqref{wallcross} as the right wall-crossing constraint.
\subsection{Seiberg-Witten with $N_f = 1$} Recall we have vanishing cycles $\gamma_1, \gamma_2, \gamma_3$, with the single relation in $\Gamma$
\begin{equation*}
\gamma_1 + \gamma_2 + \gamma_3 = 0
\end{equation*}
and intersection products 
\begin{equation*}
\bra \gamma_1, \gamma_2\ket = \bra \gamma_2, \gamma_3\ket = \bra \gamma_3, \gamma_1\ket = 1. 
\end{equation*}
Initially the BPS rays are given by
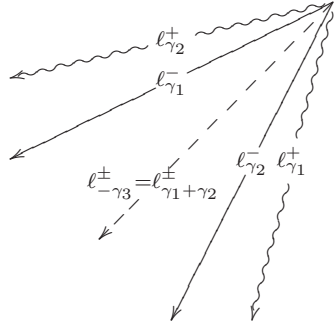
\begin{center}
\centerline{
\xymatrix{   &   &   &   &*=0{}\ar@{~>}[dllll]!U|{\ell^+_{\gamma_2}}\ar[ddllll]!U|{\ell^-_{\gamma_1}}\ar@{-->}[dddlll]!RRRUUU|>>>>>>>>>>{\ell^{\pm}_{-\gamma_3} = \ell^{\pm}_{\gamma_1 + \gamma_2}}\ar[ddddll]!U|{\ell^-_{\gamma_2}}\ar@{~>}[ddddl]!U|{\ell^+_{\gamma_1}}         \\
                 &   &   &   &           \\
                 &   &   &   &           \\
                 &   &   &   &           \\
                 &   &   &   &}}
\end{center} 
The wall-crossing formula for $N_f = 1$ is a refinement of the $N_f = 0$ case we have seen above. We highlight how this refinement happens in a specific example, namely $\Omega(\gamma_1 - \gamma_3 + \gamma_2, u^-) = -2$, starting from the GMN side. Notice that the ``restriction to effective integrals" from the $N_f=0$ case still holds, when applied to the computation of $\Omega(a\gamma_1 + b(- \gamma_3) + c\gamma_2, u^-)$ with $a, b, c$ positive. Accordingly, we will consider the diagrams
\begin{center}
\centerline{
\xymatrix{*+[F]{\gamma_1} \ar[r] & \gamma_2 \ar[r] & -\gamma_3 &   & *+[F]{\gamma_1} \ar[r] &  -\gamma_3 \ar[r] & \gamma_2\\
                   &          & \gamma_2  &   *+[F]{\gamma_1} \ar[l]\ar[r] &  -\gamma_3 &   &  }}
\end{center} 
We will write $\sigma = \sigma(\gamma_1 - \gamma_3 + \gamma_2)$. For the first diagram we have 
\begin{align*}
\W &= (-1)^3\sigma(\gamma_1)\gamma_1 \bra \gamma_1, \sigma(\gamma_2)\gamma_2\ket \bra \gamma_2, -\sigma(-\gamma_3)\gamma_3\ket\\ 
&= (-1)^4 (-1)^{\bra \gamma_1, \gamma_2 \ket} (-1)^{\bra \gamma_1 + \gamma_2, -\gamma_3 \ket}\sigma \gamma_1\\
&= - \sigma \gamma_1. 
\end{align*}
and so an integral
\begin{align*}
-\sigma\frac{\gamma_1}{4\pi i}&\int_{\ell_{\gamma_1}(u^+)}d\z_1\rho(\z, \z_1)\Xsf_{\gamma_1}(\z_1, u^+)\frac{1}{4\pi i}\int_{\ell_{\gamma_2}(u^+)}d\z_2\rho(\z_1, \z_2)\Xsf_{\gamma_2}(\z_2, u^+)\\
&\frac{1}{4\pi i}\int_{\ell_{-\gamma_3}(u^{\pm})}d\z_3\rho(\z_2, \z_3)\Xsf_{-\gamma_3}(\z_3, u^+).
\end{align*}
Pushing $\ell_{\gamma_1}(u^+)$ to $\ell_{\gamma_1}(u^-)$ and by Fubini we can rewrite this as 
\begin{align*}
-\sigma\frac{\gamma_1}{4\pi i}&\int_{\ell_{\gamma_1}(u^-)}d\z_1\rho(\z, \z_1)\Xsf_{\gamma_1}(\z_1, u^+)\frac{1}{4\pi i}\int_{\ell_{-\gamma_3}(u^{\pm})}d\z_3\Xsf_{-\gamma_3}(\z_3, u^+)\\
&\frac{1}{4\pi i}\int_{\ell_{\gamma_2}(u^+)}d\z_2\rho(\z_2, \z_3)\rho(\z_1, \z_2)\Xsf_{\gamma_2}(\z_2, u^+),
\end{align*}
As usual we want to replace $\ell_{\gamma_2}(u^+)$ with $\ell_{\gamma_2}(u^-)$. By the residue theorem we split the above integral as 
\begin{align*}
-\sigma\frac{\gamma_1}{4\pi i}&\int_{\ell_{\gamma_1}(u^-)}d\z_1\rho(\z, \z_1)\Xsf_{\gamma_1}(\z_1, u^+)\frac{1}{4\pi i}\int_{\ell_{-\gamma_3}(u^{\pm})}d\z_3\rho(\z_2, \z_3)\Xsf_{-\gamma_3}(\z_3, u^+)\\
&\frac{1}{4\pi i}\int_{\ell_{\gamma_2}(u^-)}d\z_2\rho(\z_1, \z_2)\Xsf_{\gamma_2}(\z_2, u^+)
\end{align*}
plus the residue terms
\begin{equation}\label{Nf1_I}
-\sigma\frac{\gamma_1}{4\pi i}\int_{\ell_{\gamma_1}(u^-)}d\z_1\rho(\z, \z_1)\Xsf_{\gamma_1 + \gamma_2}(\z_1, u^+)\frac{1}{4\pi i}\int_{\ell_{-\gamma_3}(u^{\pm})}d\z_3\rho(\z_1, \z_3)\Xsf_{-\gamma_3}(\z_3, u^+)
\end{equation}
and
\begin{equation}\label{Nf1_II}
-\sigma\frac{\gamma_1}{4\pi i}\int_{\ell_{\gamma_1}(u^-)}d\z_1\rho(\z, \z_1)\Xsf_{\gamma_1}(\z_1, u^+)\frac{1}{4\pi i}\int_{\ell_{-\gamma_3}(u^{\pm})}d\z_3\rho(\z_1, \z_3)\Xsf_{-\gamma_3 + \gamma_2}(\z_3, u^+)
\end{equation}
(see picture below).
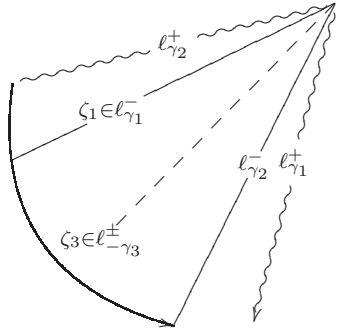
\begin{center}
\centerline{
\xymatrix{   &   &   &   &*=0{}\ar@{~}[dllll]!U|{\ell^+_{\gamma_2}}\ar@{-}[ddllll]!U|>>>>>>>>>>>>>>{\z_1\in\ell^-_{\gamma_1}}\ar@{--}[dddlll]!RRRUUU|>{\z_3\in\ell^{\pm}_{-\gamma_3}}\ar@{-}[ddddll]!U|{\ell^-_{\gamma_2}}\ar@{~>}[ddddl]!U|{\ell^+_{\gamma_1}}         \\
                 *=0{}\ar@/_2pc/@{->}[dddrr]&   &   &   &           \\
                 &   &   &   &           \\
                 &   &   &   &           \\
                 &   &  *=0{}&   &}}
\end{center} 
The term \eqref{Nf1_II} is just the same as
\begin{equation*}
-\sigma\frac{\gamma_1}{4\pi i}\int_{\ell_{\gamma_1}(u^-)}d\z_1\rho(\z, \z_1)\Xsf_{\gamma_1}(\z_1, u^+)\frac{1}{4\pi i}\int_{\ell_{-\gamma_3 + \gamma_2}(u^{-})}d\z_3\rho(\z_1, \z_3)\Xsf_{-\gamma_3 + \gamma_2}(\z_3, u^+)
\end{equation*}
and so does not give a top order correction. However the integral \eqref{Nf1_I} looks different: in the usual approach, we would need to push the first integration ray $\ell_{\gamma_1}(u^-)$ to $\ell_{\gamma_1 + \gamma_2}(u^-)$, but in the present situation $\ell_{\gamma_1 + \gamma_2}(u^-) = \ell_{-\gamma_3}(u^{\pm})$, which coincides with the second integration ray, so the resulting integral is not well defined! We will see in a moment how this new difficulty is resolved. 

Passing to the diagram
\begin{center}
\centerline{
\xymatrix{*+[F]{\gamma_1} \ar[r] &  -\gamma_3 \ar[r] & \gamma_2}}
\end{center} 
we have
\begin{align*}
\W &= (-1)^3\sigma(\gamma_1)\gamma_1 \bra \gamma_1, -\sigma(-\gamma_3)\gamma_3\ket \bra -\gamma_3, \sigma(\gamma_2)\gamma_2\ket\\
&= (-1)^3 (-1)^{\bra \gamma_1, -\gamma_3\ket} (-1)^{\bra \gamma_1 - \gamma_3, \gamma_2\ket}\sigma \gamma_1\\
&= \sigma \gamma_1
\end{align*}
and by a first application of the Fubini and residue theorems one checks that the only top degree contribution can come from the integral
\begin{equation*}
\sigma\frac{\gamma_1}{4\pi i} \int_{\ell_{\gamma_1 - \gamma_3}(u^-)}d\z_2\rho(\z, \z_2)\Xsf_{\gamma_1 - \gamma_3}(\z_2, u^+)\frac{1}{4\pi i}\int_{\ell_{\gamma_2}(u^+)}d\z_3\rho(\z_2, \z_3)\Xsf_{\gamma_2}(\z_3, u^+).
\end{equation*}
This indeed gives a residue term
\begin{equation*}
\sigma\frac{\gamma_1}{4\pi i} \int_{\ell_{\gamma_1 - \gamma_3 + \gamma_2}(u^-)}d\z_2\rho(\z, \z_2)\Xsf_{\gamma_1 - \gamma_3 + \gamma_2}(\z_2, u^+)
\end{equation*}
which requires a contribution of $-\sigma$ to $f^{\gamma_1 - \gamma_3 + \gamma_2}(u^-)$ for cancellation. Finally, we consider
\begin{center}
\centerline{
\xymatrix{\gamma_2  &   *+[F]{\gamma_1} \ar[l]\ar[r] &  -\gamma_3}}
\end{center} 
with
\begin{align*}
\W &=  (-1)^3\sigma(\gamma_1)\gamma_1 \bra \gamma_1, \sigma(\gamma_2)\gamma_2\ket \bra \gamma_1, -\sigma(-\gamma_3)\gamma_3\ket\\
&= (-1)^3 (-1)^{\bra \gamma_1, \gamma_2 \ket} (-1)^{\bra \gamma_1 + \gamma_2, -\gamma_3 \ket}\sigma\gamma_1\\
&= \sigma \gamma_1 
\end{align*}
and a corresponding integral
\begin{align*}
\sigma\frac{\gamma_1}{4\pi i}&\int_{\ell_{\gamma_1}(u^+)}d\z_1\rho(\z, \z_1)\Xsf_{\gamma_1}(\z_1, u^+)\frac{1}{4\pi i}\int_{\ell_{\gamma_2}(u^+)}d\z_2\rho(\z_1, \z_2)\Xsf_{\gamma_2}(\z_2, u^+)\\
&\frac{1}{4\pi i}\int_{\ell_{-\gamma_3}(u^{\pm})}d\z_3\rho(\z_1, \z_3)\Xsf_{-\gamma_3}(\z_3, u^+).
\end{align*}
Pushing $\ell_{\gamma_1}(u^+)$ to $\ell_{\gamma_1}(u^-)$ gives a principal term
\begin{align*}
\sigma\frac{\gamma_1}{4\pi i}&\int_{\ell_{\gamma_1}(u^-)}d\z_1\rho(\z, \z_1)\Xsf_{\gamma_1}(\z_1, u^+)\frac{1}{4\pi i}\int_{\ell_{\gamma_2}(u^+)}d\z_2\rho(\z_1, \z_2)\Xsf_{\gamma_2}(\z_2, u^+)\\
&\frac{1}{4\pi i}\int_{\ell_{-\gamma_3}(u^{\pm})}d\z_3\rho(\z_1, \z_3)\Xsf_{-\gamma_3}(\z_3, u^+)
\end{align*}
plus a residue
\begin{equation*}
\sigma\frac{\gamma_1}{4\pi i} \int_{\ell_{\gamma_1 - \gamma_3}(u^-)}d\z_3\rho(\z, \z_3)\Xsf_{\gamma_1 - \gamma_3}(\z_3, u^+)\frac{1}{4\pi i}\int_{\ell_{\gamma_2}(u^+)}d\z_2\rho(\z_3, \z_3)\Xsf_{\gamma_2}(\z_2, u^+),
\end{equation*}
which requires a contribution of $-\sigma$ to $f^{\gamma_1 - \gamma_3 + \gamma_2}(u^-)$. Finally pushing $\ell_{\gamma_2}(u^+)$ to $\ell_{\gamma_2}(u^-)$ in the principal term gives in turn a ``singular" residue 
\begin{equation}\label{Nf1_III}
\sigma\frac{\gamma_1}{4\pi i}\int_{\ell_{\gamma_1}(u^-)}d\z_1\rho(\z, \z_1)\Xsf_{\gamma_1 + \gamma_2}(\z_1, u^+)\frac{1}{4\pi i}\int_{\ell_{-\gamma_3}(u^{\pm})}d\z_3\rho(\z_1, \z_3)\Xsf_{-\gamma_3}(\z_3, u^+)
\end{equation}
which precisely cancels out the integral \eqref{Nf1_I}. Thus the final result for $f^{\gamma_1 - \gamma_3 + \gamma_2}(u^-)$ is $-2\sigma$, which gives $\Omega(\gamma_1 - \gamma_3 + \gamma_2, u^-) = -2$.\\

\noindent\textbf{Singular integrals and comparison with JS.} It is especially interesting to compare with computations with the JS formula in this case, since as we explained the GMN diagrams 
\begin{center}
\centerline{
\xymatrix{*+[F]{\gamma_1} \ar[r] & \gamma_2 \ar[r] & -\gamma_3 &   & \gamma_2  &   *+[F]{\gamma_1} \ar[l]\ar[r] &  -\gamma_3}}
\end{center} 
do not give a definite numerical numerical contribution by themselves (due to the ``singular integrals" \eqref{Nf1_I}, \eqref{Nf1_III}). We write schematically
\begin{center}
\centerline{
\xymatrix{*+[F]{\gamma_1} \ar[r] & \gamma_2 \ar[r] & -\gamma_3 \sim \textrm{sing}\\
\gamma_2  &   *+[F]{\gamma_1} \ar[l]\ar[r] &  -\gamma_3 \sim -\sigma - \textrm{sing}}}
\end{center} 
where sing is the ``value" of the integral \eqref{Nf1_I}. In GMN theory this remains undetermined, and cancellation is enough to obtain the correct result. In the JS formula the singularity is spread equally between the two diagrams, and each weighs $-\frac{\sigma}{2}$. Indeed the $\S$ symbols are given by
\begin{equation*}
\begin{matrix}
\S(\gamma_1, \gamma_2, -\gamma_3; s, w) = 0,\,\,\,\S(\gamma_1, -\gamma_3, \gamma_2; s, w) = 1,\,\,\,\S(-\gamma_3, \gamma_1, \gamma_2; s, w) = 0,\\
\S(\gamma_2, \gamma_1, -\gamma_3; s, w) = -1,\,\,\,\S(\gamma_2, -\gamma_3, \gamma_1; s, w) = 1,\,\,\,\S(-\gamma_3, \gamma_2, \gamma_1; s, w) = -1.
\end{matrix}
\end{equation*}  
Since $Z(\gamma_1 + \gamma_2) = Z(-\gamma_3)$, the $\U$ symbols are then given by
\begin{equation*}
\begin{matrix}
\U(\gamma_1, \gamma_2, -\gamma_3; s, w) = -\frac{1}{2},\,\,\,\U(\gamma_1, -\gamma_3, \gamma_2; s, w) = 1,\,\,\,\U(-\gamma_3, \gamma_1, \gamma_2; s, w) = -\frac{1}{2},\\
\U(\gamma_2, \gamma_1, -\gamma_3; s, w) = -\frac{1}{2},\,\,\,\U(\gamma_2, -\gamma_3, \gamma_1; s, w) = 1,\,\,\,\U(-\gamma_3, \gamma_2, \gamma_1; s, w) = -\frac{1}{2}.
\end{matrix}
\end{equation*}
Using this we can summarize the JS computations, twisted by 
\begin{equation*}
\left(\prod (-1)^{\bra \alpha_i, \alpha_j\ket}\right)^{-1}\left(\prod \sigma(\alpha_k)\right) = -\sigma
\end{equation*} 
as follows. Each orientation of the diagram
\begin{center}
\centerline{
\xymatrix{\gamma_1 \ar@{-}[r] & \gamma_2 \ar@{-}[r] & -\gamma_3}}
\end{center} 
contributes $-\frac{\sigma}{8}$ to $f^{\gamma_1 - \gamma_3 + \gamma_2}$, and so the diagram contributes $-\frac{\sigma}{2}$. On the other hand we have
\begin{center}
\centerline{
\xymatrix{\gamma_1 \ar[r] & -\gamma_3 \ar[r] & \gamma_2 \sim -\frac{\sigma}{4}&   & \gamma_1 \ar[r] & -\gamma_3 &  \gamma_2 \sim -\frac{3\sigma}{8}\ar[l]\\
              \gamma_1  & -\gamma_3 \ar[l]\ar[r] & \gamma_2 \sim -\frac{\sigma}{8}&   & \gamma_1 &  -\gamma_3  \ar[l] & \gamma_2 \sim -\frac{\sigma}{4}\ar[l]}}
\end{center} 
So the JS contribution of unoriented, labelled diagram 
\begin{center}
\centerline{
\xymatrix{\gamma_1 \ar@{-}[r] & -\gamma_3 \ar@{-}[r] & \gamma_2}}
\end{center}
is $-\sigma$, matching the GMN one, as predicted by our conjecture. Finally, all orientations of the diagram
\begin{center}
\centerline{
\xymatrix{\gamma_2 \ar@{-}[r] & \gamma_1 \ar@{-}[r] & -\gamma_3}}
\end{center}
equally contribute $-\frac{\sigma}{8}$ in the JS theory, and so $-\frac{\sigma}{2}$ in total. 

The upshot is that in this example there is a unique value we can assign to the singular term so that Conjecture \ref{mainConj} still makes sense and is verified, namely $\textrm{sing} = -\frac{\sigma}{2}$. We expect that this is always the case when we encounter singular GMN integrals, and we extend Conjecture \ref{mainConj} to include this claim. We will see another example of this behaviour in the next computation.
\subsection{Seiberg-Witten with $N_f = 2$}\label{Nf2Sect} We denote by $\gamma^1_1, \gamma^2_1$ and $\gamma^1_2, \gamma^2_2$ the vanishing cycles for the Hitchin fibration, and set $\sigma = \sigma(\gamma^1_1 + \gamma^2_1 + \gamma^1_2 + \gamma^2_2)$. There are two relations in $\Gamma$,
\begin{equation*}
\gamma^1_1 - \gamma^2_1  = \gamma^1_2 - \gamma^2_2 = 0,
\end{equation*} 
and the nonvanishing intersection products are given by 
\begin{equation*}
\bra \gamma^i_1, \gamma^j_2\ket = 1. 
\end{equation*}
Initially the BPS rays are given by
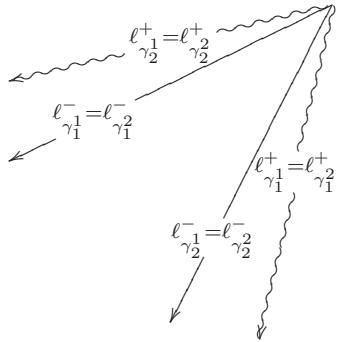
\begin{center}
\centerline{
\xymatrix{   &   &   &   &*=0{}\ar@{~>}[dllll]!U|{\ell^+_{\gamma^1_2} = \ell^+_{\gamma^2_2}}\ar[ddllll]!U|>>>>>>>>>>>>{\ell^-_{\gamma^1_1} = \ell^-_{\gamma^2_1}}\ar[ddddll]!U|>>>>>>>>>>>>{\ell^-_{\gamma^1_2}=\ell^-_{\gamma^2_2}}\ar@{~>}[ddddl]!DRRRR|{\ell^+_{\gamma^1_1}=\ell^+_{\gamma^2_1}}         \\
                 &   &   &   &           \\
                 &   &   &   &           \\
                 &   &   &   &           \\
                 &   &   &   &}}
\end{center}
So we can restrict to ``effective integrals" for $\Omega(a_1\gamma^1_1 + a_2\gamma^2_1 + a_3\gamma^1_2 + a_4\gamma^2_2, u^-)$ with positive $a_i$. For $\Omega(\gamma^1_1 + \gamma^2_1 + \gamma^1_2 + \gamma^2_2, u^-)$ we need to consider the diagrams
\begin{center}
\centerline{
\xymatrix{*+[F]{\gamma^1_1} \ar[r] & \gamma^1_2 \ar[r] & \gamma^2_1 \ar[r]& \gamma^2_2 & & \gamma^2_1 & \ar[l] \gamma^1_2 &*+[F]{\gamma^1_1}\ar[l]\ar[r] & \gamma^2_2 \\
*+[F]{\gamma^1_1} \ar[r] & \gamma^2_2 \ar[r] & \gamma^2_1 \ar[r]& \gamma^1_2 & & \gamma^2_1 & \ar[l] \gamma^2_2 &*+[F]{\gamma^1_1}\ar[l]\ar[r] & \gamma^1_2
}
}
\end{center}
By symmetry, the final result for $\Omega(\gamma^1_1 + \gamma^2_1 + \gamma^1_2 + \gamma^2_2, u^-)$ must be twice the sum of the contributions of the two top diagrams. In fact we (almost) already computed the upper left diagram, when dealing with 
\begin{center} 
\centerline{
\xymatrix{*+[F]{ \d} \ar[r] & \ar[r] \m \ar[r] & \d \ar[r] & 2\m}
}
\end{center}
in the $N_f = 0$ theory. The only difference is that now 
\begin{align*}
\W &= (-1)^4 \sigma(\gamma^1_1)\gamma^1_1 \bra \gamma^1_1, \sigma(\gamma^1_2)\gamma^1_2\ket \bra \gamma^1_2, \sigma(\gamma^2_1)\gamma^2_1\ket \bra \gamma^2_1, \sigma(\gamma^2_2)\gamma^2_2\ket\\
&= - (-1)^{\bra\gamma^1_1, \gamma^1_2\ket}(-1)^{\bra \gamma^1_1 + \gamma^1_2, \gamma^2_1 \ket}(-1)^{\bra \gamma^1_1 + \gamma^1_2 + \gamma^2_1 , \gamma^2_2\ket}\sigma\gamma^1_1\\
&= -\sigma \gamma^1_1
\end{align*}
and that of course we replace $2\m$ by $\gamma^2_2$. So we replace the integral \eqref{use4Nf2} with the ``singular" integral
\begin{equation}\label{Nf2_sing1}
-\sigma\frac{\gamma^1_1}{4\pi i}\int_{\ell_{\gamma^1_1 + \gamma^1_2}(u^-)}d\z_1\rho(\z, \z_1)\chi^{\sf}_{\gamma^1_1 + \gamma^1_2}(\z_1, u^+)\frac{1}{4\pi i} \int_{\ell_{\gamma^2_1}(u^-)}d\z_3\rho(\z_1, \z_3)\chi^{\sf}_{\gamma^2_1 + \gamma^2_2}(\z_3, u^+),
\end{equation}
while the analogue of \eqref{someIntegral} requires a contribution of $-\sigma$ to $f^{\gamma^1_1 + \gamma^2_1 + \gamma^1_2 + \gamma^2_2}(u^-)$. 

Passing now to the upper right diagram, we have
\begin{align*}
\W &= (-1)^4 \sigma(\gamma^1_1)\gamma^1_1 \bra \gamma^1_1, \sigma(\gamma^2_2)\gamma^2_2\ket \bra \gamma^1_1, \sigma(\gamma^1_2)\gamma^1_2\ket \bra \gamma^1_2, \sigma(\gamma^2_1)\gamma^2_1\ket\\
&= - (-1)^{\bra\gamma^1_1, \gamma^2_2\ket}(-1)^{\bra \gamma^1_1 + \gamma^2_2, \gamma^1_2\ket}(-1)^{\bra \gamma^1_1 + \gamma^2_2 + \gamma^1_2, \gamma^2_1\ket}\sigma\gamma^1_1\\
&= -\sigma \gamma^1_1
\end{align*}
and a corresponding integral
\begin{align*}
-\sigma\frac{\gamma^1_1}{4\pi i}&\int_{\ell_{\gamma^1_1}(u^+)}d\z_1\rho(\z, \z_1)\Xsf_{\gamma^1_1}(\z_1, u^+)\frac{1}{4\pi i}\int_{\ell_{\gamma^1_2}(u^+)}d\z_2\rho(\z_1, \z_2)\Xsf_{\gamma^1_2}(\z_2, u^+)\\
&\frac{1}{4\pi i}\int_{\ell_{\gamma^2_2}(u^+)}d\z_3\rho(\z_1, \z_3)\Xsf_{\gamma^2_2}(\z_3, u^+)\frac{1}{4\pi i}\int_{\ell_{\gamma^2_1}(u^+)}d\z_4\rho(\z_3, \z_4)\Xsf_{\gamma^2_1}(\z_4, u^+).
\end{align*}
Pushing $\ell_{\gamma^2_2}(u^+)$ to $\ell_{\gamma^2_2}(u^-)$ gives no residue. Then pushing $\ell_{\gamma^2_1}(u^+)$ to $\ell_{\gamma^2_1}(u^-)$ splits the integral as
\begin{align*}
-\sigma\frac{\gamma^1_1}{4\pi i}&\int_{\ell_{\gamma^1_1}(u^+)}d\z_1\rho(\z, \z_1)\Xsf_{\gamma^1_1}(\z_1, u^+)\frac{1}{4\pi i}\int_{\ell_{\gamma^1_2}(u^+)}d\z_2\rho(\z_1, \z_2)\Xsf_{\gamma^1_2}(\z_2, u^+)\\
&\frac{1}{4\pi i}\int_{\ell_{\gamma^2_2}(u^-)}d\z_3\rho(\z_1, \z_3)\Xsf_{\gamma^2_2}(\z_3, u^+)\frac{1}{4\pi i}\int_{\ell_{\gamma^2_1}(u^-)}d\z_4\rho(\z_3, \z_4)\Xsf_{\gamma^2_1}(\z_4, u^+).
\end{align*}
plus a residue term
\begin{align*}
\sigma\frac{\gamma^1_1}{4\pi i}&\int_{\ell_{\gamma^1_1}(u^+)}d\z_1\rho(\z, \z_1)\Xsf_{\gamma^1_1}(\z_1, u^+)\frac{1}{4\pi i}\int_{\ell_{\gamma^1_2}(u^+)}d\z_2\rho(\z_1, \z_2)\Xsf_{\gamma^1_2}(\z_2, u^+)\\
&\frac{1}{4\pi i}\int_{\ell_{\gamma^2_2}(u^-)}d\z_3\rho(\z_1, \z_3)\Xsf_{\gamma^2_1 + \gamma^2_2}(\z_3, u^+).
\end{align*}
It is clear that top order contributions can only come from the residue term. Using Fubini and pushing $\ell_{\gamma^1_1}(u^+)$ to $\ell_{\gamma^1_1}(u^-)$ we rewrite this as
\begin{align*}
\sigma\frac{\gamma^1_1}{4\pi i}&\int_{\ell_{\gamma^1_1}(u^-)}d\z_1\rho(\z, \z_1)\Xsf_{\gamma^1_1}(\z_1, u^+)\frac{1}{4\pi i}\int_{\ell_{\gamma^1_2}(u^+)}d\z_2\rho(\z_1, \z_2)\Xsf_{\gamma^1_2}(\z_2, u^+)\\
&\frac{1}{4\pi i}\int_{\ell_{\gamma^2_2}(u^-)}d\z_3\rho(\z_1, \z_3)\Xsf_{\gamma^2_1 + \gamma^2_2}(\z_3, u^+).
\end{align*} 
plus a residue, which we can write as
\begin{equation*}
\sigma\frac{\gamma^1_1}{4\pi i}\int_{\ell_{\gamma^1_2}(u^+)}d\z_2\rho(\z_1, \z_2)\Xsf_{\gamma^1_2}(\z_2, u^+)\frac{1}{4\pi i}\int_{\ell_{\gamma^1_1 + \gamma^2_1 + \gamma^2_2}(u^-)}d\z_3\rho(\z_1, \z_3)\Xsf_{\gamma^1_1 + \gamma^2_1 + \gamma^2_2}(\z_3, u^+).
\end{equation*}
Pushing $\ell_{\gamma^1_2}(u^+)$ to $\ell_{\gamma^1_2}(u^-)$ in the first integral gives a singular term  
\begin{equation}\label{Nf2_sing2}
\sigma\frac{\gamma^1_1}{4\pi i}\int_{\ell_{\gamma^1_1 + \gamma^1_2}(u^-)}d\z_1\rho(\z, \z_1)\Xsf_{\gamma^1_1 + \gamma^1_2}(\z_1, u^+)\frac{1}{4\pi i}\int_{\ell_{\gamma^2_2}(u^-)}d\z_3\rho(\z_1, \z_3)\Xsf_{\gamma^2_1 + \gamma^2_2}(\z_3, u^+),
\end{equation} 
while pushing $\ell_{\gamma^1_2}(u^+)$ to $\ell_{\gamma^1_2}(u^-)$ in the residue contributes $-\sigma$ to $f$. Finally combining the singular terms \eqref{Nf2_sing1} and \eqref{Nf2_sing2} contributes $\sigma$ to $f$: to see this just notice that in \eqref{Nf2_sing1} the ray $\ell_{\gamma^2_1}(u^-)$ approaches $\ell_{\gamma^1_1 + \gamma^1_2}(u^-)$ in the counterclockwise, while in \eqref{Nf2_sing2} the ray $\ell_{\gamma^2_2}(u^-)$ approaches it in the clockwise direction. So the total contribution of the upper diagrams to $f$ is $-\sigma$, and taking into account the bottom diagrams too shows $f = -2\sigma$, hence $\Omega = -2$.\\

\noindent\textbf{Singular integrals and comparison with JS.} Recall that we found
\begin{center}
\centerline{
\xymatrix{*+[F]{\gamma^1_1} \ar[r] & \gamma^1_2 \ar[r] & \gamma^2_1 \ar[r]& \gamma^2_2 \sim -\sigma + \operatorname{sing}_1\\ 
\gamma^2_1 & \ar[l] \gamma^1_2 &*+[F]{\gamma^1_1}\ar[l]\ar[r] & \gamma^2_2 \sim -\sigma + \operatorname{sing}_2}}
\end{center} 
where we write $\operatorname{sing}_1$ ($\operatorname{sing}_2$) for the ``value" of \eqref{Nf2_sing1} (respectively \eqref{Nf2_sing2}). In GMN theory these values are undetermined, and only satisfy the constraint $\operatorname{sing}_1 + \operatorname{sing}_2 = \sigma$. This ambiguity is resolved in Joyce-Song theory. We have 
\begin{equation*}
\begin{matrix}
\U(\gamma^1_1, \gamma^2_1, \gamma^1_2, \gamma^2_2) = \frac{1}{4},\,\,\,\U(\gamma^1_1, \gamma^1_2, \gamma^2_1, \gamma^2_2) = -\frac{1}{2},\,\,\,\U(\gamma^1_1, \gamma^1_2, \gamma^2_2, \gamma^2_1) = 0,\\
\U(\gamma^1_2, \gamma^1_1, \gamma^2_1, \gamma^2_2) = 0,\,\,\,\U(\gamma^1_2, \gamma^2_2, \gamma^1_1, \gamma^2_1) = -\frac{1}{4},\,\,\,\U(\gamma^1_2, \gamma^1_1, \gamma^2_2, \gamma^2_1) = \frac{1}{2}
\end{matrix} 
\end{equation*}
(with the symmetries $\gamma^1_1 \leftrightarrow \gamma^2_1$, $\gamma^1_2 \leftrightarrow \gamma^2_2$), from which we can compute (twisting by $\left(\prod (-1)^{\bra \alpha_i, \alpha_j\ket}\right)^{-1}\left(\prod \sigma(\alpha_k)\right) = -\sigma$) 
\begin{center}
\centerline{
\xymatrix{\gamma^1_1 \ar[r]& \gamma^1_2 & \ar[l]\gamma^2_1\ar[r] & \gamma^2_2\,\,\sim -4\cdot\frac{\sigma}{2^5},\,\,\,\,\,\,\,\,\,\gamma^1_1 \ar[r]& \gamma^1_2 \ar[r]& \gamma^2_1\ar[r] & \gamma^2_2\,\,\sim -4\cdot\frac{\sigma}{2^4}
}}
\end{center}
\begin{center}
\centerline{
\xymatrix{\gamma^1_1 & \ar[l]\gamma^1_2\ar[r] & \gamma^2_1 & \ar[l]\gamma^2_2\,\,\sim 4\cdot\frac{\sigma}{2^5}
}}
\end{center}
(the other orientations vanish). Similarly we have
\begin{center}
\centerline{
\xymatrix{\gamma^2_1 \ar[r]& \gamma^1_2 & \ar[l]\gamma^1_1\ar[r] & \gamma^2_2\,\,\sim -4\cdot\frac{\sigma}{2^5},\,\,\,\,\,\,\,\,\,\gamma^2_1 & \ar[l]\gamma^1_2 & \ar[l]\gamma^1_1 \ar[r] & \gamma^2_2\,\,\sim -4\cdot\frac{\sigma}{2^4}\\
\gamma^2_1 & \ar[l]\gamma^1_2\ar[r] & \gamma^1_1 & \ar[l]\gamma^2_2\,\,\sim -4\cdot\frac{\sigma}{2^5},\,\,\,\,\,\,\,\,\,\gamma^2_1 & \ar[l]\gamma^1_2\ar[r] & \gamma^1_1 \ar[r]& \gamma^2_2\,\,\sim -4\cdot\frac{\sigma}{2^4} 
}}
\end{center}
while the other orientations vanish. So for the top diagrams we get a JS result of $-\sigma$, which by symmetry is the same as the contribution of the bottom diagrams, giving a JS result of $-2\sigma$. Writing schematically the JS terms as
\begin{center}
\centerline{
\xymatrix{ \gamma^1_1 \ar@{-}[r] & \gamma^1_2 \ar@{-}[r] & \gamma^2_1 \ar@{-}[r]& \gamma^2_2 \sim -\frac{\sigma}{4}\\ 
\gamma^2_1 & \ar@{-}[l] \gamma^1_2 &\gamma^1_1\ar@{-}[l]\ar@{-}[r] & \gamma^2_2 \sim -\frac{3\sigma}{4} }}
\end{center}
we find Conjecture \ref{mainConj} holds if and only if we set $\operatorname{sing}_1 = \frac{3\sigma}{4}$, $\operatorname{sing}_2 = \frac{\sigma}{4}$. This is compatible with the constraint $\operatorname{sing}_1 + \operatorname{sing}_2 = \sigma$.\\

\noindent\textbf{Conjecture \ref{mainConj} and singular integrals.} Keeping the above examples in mind, we supplement Conjecture \ref{mainConj} stating the precise behaviour we expect for singular integrals.
\begin{enumerate}
\item[$\bullet$] Singular integrals in GMN theory should always either cancel out, as for \eqref{Nf1_I}, \eqref{Nf1_III}, or combine to give a well defined (nonsingular) integral, following the model case of \eqref{Nf2_sing1}, \eqref{Nf2_sing2}. Notice however that this process will involve singular integrals arising from different diagrams. 
\item[$\bullet$] This behaviour gives a set of contraints on the ``values" of the singular integrals, which however remain undetermined. We conjecture that there is a unique way of specifying preferred values for all singular integrals, which are compatible with the GMN constraints, and for which Conjecture 1 holds (that is, the value of a diagram in GMN and JS theories is the same).
\end{enumerate}
\subsection{Seiberg-Witten with $N_f = 3$} Initially we have BPS rays
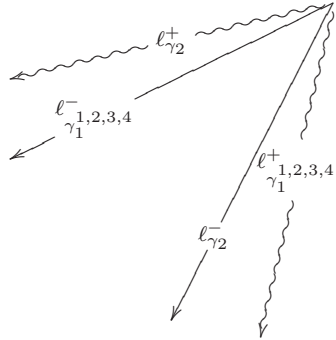
\begin{center}
\centerline{
\xymatrix{   &   &   &   &*=0{}\ar@{~>}[dllll]!U|{\ell^+_{\gamma_2}}\ar[ddllll]!U|>>>>>>>>>>>>{\ell^-_{\gamma^{1,2,3,4}_1}}\ar[ddddll]!U|>>>>>>>>>>>>{\ell^-_{\gamma_2}}\ar@{~>}[ddddl]!DRRRR|{\ell^+_{\gamma^{1,2,3,4}_1}}         \\
                 &   &   &   &           \\
                 &   &   &   &           \\
                 &   &   &   &           \\
                 &   &   &   &}}
\end{center}
where $\gamma^{1,2,3,4}_1$ have the same image in $\Gamma$, and $\bra \gamma^i_1, \gamma_2\ket = 1$. We study the index $\Omega(\sum_i \gamma^i_1 + 2\gamma_2) = -2$. We would need to consider the diagrams 
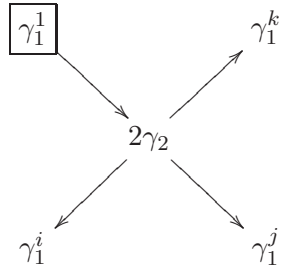
\begin{center}
\centerline{
\xymatrix{ *+[F]{ \gamma^1_1}\ar[dr]    &                                                & \gamma^k_1  \\
                                                         &   2\gamma_2\ar[dl]\ar[dr]\ar[ur]   &                             \\
                        \gamma^i_1                &                                                & \gamma^j_1 }
}
\end{center}
(a single diagram up to symmetry)
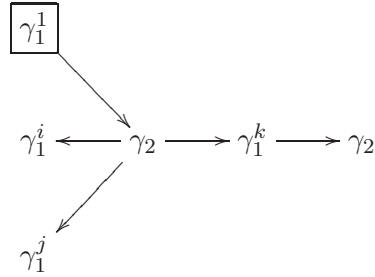
\begin{center}
\centerline{
\xymatrix{ *+[F]{ \gamma^1_1}\ar[dr]    &                                        &  &\\
               \gamma^i_1                         &   \gamma_2\ar[r]\ar[l]\ar[dl]& \gamma^k_1\ar[r] &\gamma_2\\
                \gamma^j_1                        &                                        & &}
}
\end{center}
(3 distinguished diagrams)
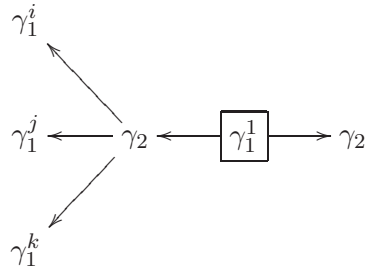
\begin{center}
\centerline{
\xymatrix{  \gamma^i_1    &                                          &  &\\
                \gamma^j_1    &   \gamma_2\ar[ul]\ar[l]\ar[dl] & *+[F]{ \gamma^1_1}\ar[l]\ar[r] &\gamma_2\\
                \gamma^k_1    &                                         & &}
}
\end{center}
(a single diagram)
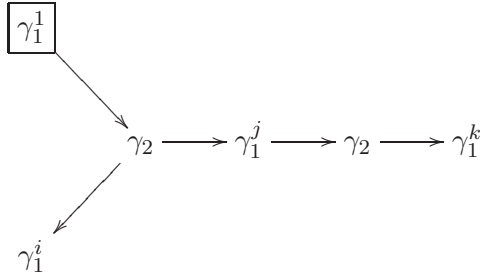
\begin{center}
\centerline{
\xymatrix{ *+[F]{ \gamma^1_1}\ar[dr]    &                                 &  & &\\
                                                         &   \gamma_2\ar[dl]\ar[r]& \gamma^j_1\ar[r] &\gamma_2\ar[r]&\gamma^k_1\\
                                   \gamma^i_1    &                                 & & &}
}
\end{center}
(6 distinguished diagrams). We only discuss the first diagram in detail. We have
\begin{align*}
\W &= (-1)^5 \sigma(\gamma^1_1)\gamma^1_1 \bra \gamma^1_1, \frac{1}{2}\sigma(2\gamma_2)\gamma_2\ket \bra 2\gamma_2, \sigma(\gamma^i_1)\gamma^i_1\ket \bra 2\gamma_2, \sigma(\gamma^j_1)\gamma^j_1\ket \bra 2\gamma_2, \sigma(\gamma^k_1)\gamma^k_1\ket\\
&= (-1)^{\bra \gamma^1_1, 2\gamma_2\ket}(-1)^{\bra \gamma^1_1 + 2\gamma_2, \gamma^i_1 \ket}(-1)^{\bra \gamma^1_1 + 2\gamma_2 + \gamma^i_1, \gamma^j_1\ket}(-1)^{\bra \gamma^1_1 + 2\gamma_2 + \gamma^i_1 + \gamma^j_1, \gamma^k_1\ket}4\sigma\gamma^1_1\\
&= 4\sigma\gamma^1_1
\end{align*}
(where we put $\sigma = \sigma(\sum_i \gamma^i_1 + 2\gamma_2)$) and an integral
\begin{align*}
4\sigma\frac{\gamma^1_1}{4\pi i}&\int_{\ell_{\gamma^1_1}(u^+)}d\z_1\rho(\z, \z_1)\Xsf_{\gamma^1_1}(\z_1, u^+)\frac{1}{4\pi }\int_{\ell_{\gamma_2}(u^+)}d\z_2\rho(\z_1, \z_2)\Xsf_{2\gamma_2}(\z_2, u^+)\\
&\frac{1}{4\pi i}\int_{\ell_{\gamma^2_1}(u^+)}d\z_3\rho(\z_2, \z_3)\Xsf_{\gamma^2_1}(\z_3, u^+) \frac{1}{4\pi i}\int_{\ell_{\gamma^3_1}(u^+)}d\z_4\rho(\z_2, \z_4)\Xsf_{\gamma^3_1}(\z_4, u^+) \\
&\frac{1}{4\pi i}\int_{\ell_{\gamma^4_1}(u^+)}d\z_5\rho(\z_2, \z_5)\Xsf_{\gamma^4_1}(\z_5, u^+).  
\end{align*}
Pushing $\ell_{\gamma^1_1}(u^+)$ to $\ell_{\gamma^1_1}(u^-)$, then $\ell_{\gamma_2}(u^+)$ to $\ell_{\gamma_2}(u^-)$ gives a principal part
\begin{align*}
4\sigma\frac{\gamma^1_1}{4\pi i}&\int_{\ell_{\gamma^1_1}(u^-)}d\z_1\rho(\z, \z_1)\Xsf_{\gamma^1_1}(\z_1, u^+)\frac{1}{4\pi }\int_{\ell_{\gamma_2}(u^-)}d\z_2\rho(\z_1, \z_2)\Xsf_{2\gamma_2}(\z_2, u^+)\\
&\frac{1}{4\pi i}\int_{\ell_{\gamma^2_1}(u^+)}d\z_3\rho(\z_2, \z_3)\Xsf_{\gamma^2_1}(\z_3, u^+) \frac{1}{4\pi i}\int_{\ell_{\gamma^3_1}(u^+)}d\z_4\rho(\z_2, \z_4)\Xsf_{\gamma^3_1}(\z_4, u^+) \\
&\frac{1}{4\pi i}\int_{\ell_{\gamma^4_1}(u^+)}d\z_5\rho(\z_2, \z_5)\Xsf_{\gamma^4_1}(\z_5, u^+),  
\end{align*} 
plus a residue
\begin{align*}
4\sigma\frac{\gamma^1_1}{4\pi i}&\int_{\ell_{\gamma^1_1 + 2\gamma_2}(u^-)}d\z_1\rho(\z, \z_1)\Xsf_{\gamma^1_1 + 2\gamma_2}(\z_1, u^+)\frac{1}{4\pi i}\int_{\ell_{\gamma^2_1}(u^+)}d\z_3\rho(\z_1, \z_3)\Xsf_{\gamma^2_1}(\z_3, u^+)\\
&\frac{1}{4\pi i}\int_{\ell_{\gamma^3_1}(u^+)}d\z_4\rho(\z_1, \z_4)\Xsf_{\gamma^3_1}(\z_4, u^+)\frac{1}{4\pi i}\int_{\ell_{\gamma^4_1}(u^+)}d\z_5\rho(\z_1, \z_5)\Xsf_{\gamma^4_1}(\z_5, u^+).
\end{align*}
Clearly the only top order contributions can come from the residue term. Pushing $\ell_{\gamma^i_1}(u^+)$ to $\ell_{\gamma^i_1}(u^-)$ (for $i = 2, 3, 4$) we get a term
\begin{equation*}
(-1)^3  4\sigma\frac{\gamma^1_1}{4\pi i}\int_{\ell_{\gamma^1_1 + \gamma^2_1 + \gamma^3_1 + \gamma^4_1 + 2\gamma_2}(u^-)}d\z_1\rho(\z, \z_1)\Xsf_{\gamma^1_1 + \gamma^2_1 + \gamma^3_1 + \gamma^4_1 + 2\gamma_2}(\z_1, u^+),
\end{equation*}
and so a contribution of $4\sigma$ to $f$. 

On the JS side, we have 
\begin{equation*}
\begin{matrix}
\U(\gamma^1_1, \gamma^2_1, \gamma^3_1, \gamma^4_1, 2\gamma_2; s, w) = \frac{1}{24},\,\,\,\U(\gamma^1_1, \gamma^2_1, \gamma^3_1, 2\gamma_2,\gamma^4_1; s, w) = -\frac{1}{6},\\
\U(\gamma^1_1, \gamma^2_1, 2\gamma_2,\gamma^3_1, \gamma^4_1; s, w) = \frac{3}{2},\,\,\,\U(\gamma^1_1, 2\gamma_2,\gamma^2_1, \gamma^3_1, \gamma^4_1; s, w) = -\frac{1}{6},\\
\U(2\gamma_2, \gamma^1_1, \gamma^2_1, \gamma^3_1, \gamma^4_1; s, w) = -\frac{1}{24},
\end{matrix}
\end{equation*}
from which we get, summing over permutations of $\gamma^1_1, \gamma^2_1, \gamma^3_1, \gamma^4_1$, and twisting by $\left(\prod (-1)^{\bra \alpha_i, \alpha_j\ket}\right)^{-1}\left(\prod \sigma(\alpha_k)\right) = \sigma$,
\begin{center}
\centerline{
\xymatrix{ \gamma^i_1\ar[dr]    &                                                & \gamma^l_1\ar[dl] &   &      
              \gamma^i_1\ar[dr]    &                                                & \gamma^l_1\\
                                                         &   2\gamma_2                             &  \sim \frac{\sigma}{4}       &  &
                                                         &   2\gamma_2\ar[ur]                     & \sim \sigma                \\
                        \gamma^j_1\ar[ur]       &                                                & \gamma^k_1\ar[ul]  &  &
                        \gamma^j_1\ar[ur]       &                                                & \gamma^k_1\ar[ul]\\
}}
\end{center}
\begin{center}
\centerline{
\xymatrix{& &\gamma^i_1 & &\gamma^l_1 & &\\
                                                                        & & &2\gamma^2_2\ar[ur]\ar[dr]\ar[dl]\ar[ul]& \sim \frac{\sigma}{4} & &\\
                                                                        & &\gamma^j_1 & &\gamma^k_1 & &}}
\end{center}
\begin{center}
\centerline{
\xymatrix{               \gamma^i_1\ar[dr]    &                                                & \gamma^l_1&   &      
              \gamma^i_1\ar[dr]    &                                                & \gamma^l_1\\
                                                         &   2\gamma_2\ar[ur]\ar[dr]             &  \sim \frac{3\sigma}{2}       &  &
                                                         &   2\gamma_2\ar[ur]\ar[dr]\ar[dl]    & \sim \sigma                \\
                        \gamma^j_1\ar[ur]      &                                                  & \gamma^k_1  &  &
                        \gamma^j_1               &                                                   & \gamma^k_1\\                                                                        
}}
\end{center}

\vspace{.5cm}
Dipartimento di Matematica ``F. Casorati"\\
Universit\`a di Pavia, via Ferrata 1, 27100 Pavia ITALY\\
\tt{jacopo.stoppa@unipv.it}
\end{document}